\documentclass{article}
\usepackage{a4wide}

\usepackage{graphicx}
\usepackage{amsmath}
\usepackage{amssymb}
\usepackage{authblk}

\usepackage{latexsym}
\usepackage{charter}
\usepackage{tikz}
\usepackage{subfigure}
\usepackage{float}
\usepackage{xspace}
\usepackage{stmaryrd}

\usepackage{proof}
\usepackage{url}
\usepackage[all]{xy}

\allowdisplaybreaks[4]

\usepackage{chngcntr}
\counterwithout{equation}{section}

\newcommand{\M}{\mathcal{M}}

\newcommand{\N}{\mathcal{N}}

\newcommand{\lra}{\leftrightarrow}

\renewcommand{\phi}{\varphi}

\newcommand{\Kh}{\mathsf{K\!h }}
\newcommand{\K}{\mathsf{K}}
\newcommand{\hK}{\widehat{\K}}

\newcommand{\KhL}{\mathbf{KhL}}

\newcommand{\Prop}{\mathbf{P}}
\newcommand{\PL}{\mathbf{PL}^\tensor}
\newcommand{\InqL}{\mathbf{InqB}}
\newcommand{\InqKhL}{\mathbf{InqKhL}}

\newcommand{\PD}{\mathbf{PD}}
\newcommand{\PID}{\mathbf{PID}}
\newcommand{\PDV}{\mathbf{PD}^{\vee}}

\newcommand{\LKhP}{\mathbf{PALKh\Pi}}
\newcommand{\LELPi}{\mathbf{EL\Pi}}
\newcommand{\LEL}{\mathbf{EL}}
\newcommand{\LPALPi}{\mathbf{PAL\Pi}}

\newcommand{\PLT}{\mathbf{PL}^{\tensor^k_n}}
\newcommand{\LKhPT}{\mathbf{PALKh\Pi G}}
\newcommand{\arank}{\mathbf{ar}}

\newcommand{\R}{R}
\newcommand{\RR}{\textit{RL}}

\newcommand{\V}{V}
\renewcommand{\S}{S}

\renewcommand{\emptyset}{\varnothing}

\newcommand{\SKhPALPi}{\mathsf{S5KhPAL\Pi^+}}
\newcommand{\SKhPALPiT}{\mathsf{S5KhPAL\Pi^+ G}}
\newcommand{\SFive}{\textsf{S5}}

\newcommand{\TAUT}{\ensuremath{\mathtt{TAUT}}}
\newcommand{\generalization}{\ensuremath{\mathtt{GEN_\forall}}}
\newcommand{\NECK}{\ensuremath{\mathtt{NEC_\K}}}

\newcommand{\DISTK}{\ensuremath{\mathtt{DIST_\K}}}

\newcommand{\DISTA}{\ensuremath{\mathtt{DIST_\forall}}}
\newcommand{\AxTK}{\ensuremath{\mathtt{T_{\K}}}}
\newcommand{\AxTransK}{\ensuremath{\mathtt{4_{\K}}}}
\newcommand{\AxEucK}{\ensuremath{\mathtt{5_{\K}}}}

\newcommand{\Barcan}{\ensuremath{\mathtt{BC}}}
\newcommand{\Anp}{\ensuremath{\mathtt{[\, ]_{p}}}}

\newcommand{\Anid}{\ensuremath{\mathtt{[\, ]_{pre}}}}
\newcommand{\tensoror}{\ensuremath{\mathtt{Rd\tensor}}}
\newcommand{\Ancirc}{\ensuremath{\mathtt{[\, ]_{\bigcirc}}}}

\newcommand{\AnK}{\ensuremath{\mathtt{[\, ]_{\K}}}}
\newcommand{\Anforall}{\ensuremath{\mathtt{[\, ]_{\forall}}}}
\newcommand{\Anexists}{\ensuremath{\mathtt{[\, ]_{\exists}}}}
\newcommand{\subforall}{\ensuremath{\mathtt{SUB_{\forall}}}}
\newcommand{\SU}{\ensuremath{\mathtt{SU}}}
\newcommand{\MP}{\ensuremath{\mathtt{MP}}}
\newcommand{\KKhp}{\ensuremath{\mathtt{KKhp}}}
\newcommand{\KhK}{\ensuremath{\mathtt{KhK}}}
\newcommand{\KhC}{\ensuremath{\mathtt{Kh_{\land}}}}
\newcommand{\KhD}{\ensuremath{\mathtt{Kh_{\lor}}}}
\newcommand{\KhI}{\ensuremath{\mathtt{Kh_{\to}}}}
\newcommand{\Kht}{\ensuremath{\mathtt{Kh_{\tensor}}}}

\newcommand{\Khbot}{\ensuremath{\mathtt{Kh_{\bot}}}}
\newcommand{\AxTransKh}{\ensuremath{\mathtt{4_\Kh}}}
\newcommand{\AxEucKh}{\ensuremath{\mathtt{5_\Kh}}}

\newcommand{\rRE}{\ensuremath{\mathtt{rRE}}}

\newcommand{\Khgt}{\ensuremath{\mathtt{Kh_{\tensor^k_n}}}}
\newcommand{\gt}{\ensuremath{\mathtt{Rd\tensor^k_n}}}

\newcommand{\SFivePiplus}{\mathsf{S5\Pi^+}}

\newcommand{\tensor}{\otimes}
\newcommand{\dep}{=\!\!}

\newcommand{\citep}[1]{\cite{#1}}
\newtheorem{theorem}{Theorem}
\newtheorem{corollary}[theorem]{Corollary}
\newtheorem{lemma}[theorem]{Lemma}
\newtheorem{proposition}[theorem]{Proposition}
\newtheorem{definition}[theorem]{Definition}

\newtheorem{remark}{Remark}

\newenvironment{proof} {\textsc{Proof}\quad} {\hfill $\blacksquare$\\}

\begin{document}


  \title{An Epistemic Interpretation of Tensor Disjunction}
  \author{Haoyu Wang}
  \author{Yanjing Wang}
    \affil{Department of Philosophy, Peking University}
  \author{Yunsong Wang}
    \affil{ILLC, University of Amsterdam}
\maketitle

  \begin{abstract}
  This paper aims to give an epistemic interpretation to the \textit{tensor disjunction} in dependence logic, through a rather surprising connection to the so-called \textit{weak disjunction} in Medvedev's early work on intermediate logic under the Brouwer-Heyting-Kolmogorov (BHK)-interpretation. We expose this connection in the setting of inquisitive logic with tensor disjunction $\InqL^\tensor$ discussed by \cite{CiardelliB19}, but from an epistemic perspective. More specifically, we translate the propositional formulae of $\InqL^\tensor$ into modal formulae in a powerful epistemic language of \textit{knowing how} following the proposal by \cite{Wang2021,Wang321}. We give a complete axiomatization of the logic of our full language based on Fine's axiomatization of S5 modal logic with propositional quantifiers. Finally we generalize the tensor operator with parameters $k$ and $n$, which intuitively captures the epistemic situation that one knows $n$ potential answers to $n$ questions and is sure $k$ answers of them must be correct. The original tensor disjunction is the special case when $k=1$ and $n=2$. We show that the generalized tensor operators do not increase the expressive power of our logic, the inquisitive logic and propositional dependence logic, though most of these generalized tensors are not uniformly definable in these logics, except in our dynamic epistemic logic of knowing how. 
  \end{abstract}
  

\section{Introduction}

As a rapidly growing field of research, \textit{Dependence Logic} studies reasoning patterns expressed by logical languages extended with (in)dependence atoms (cf. e.g., \cite{sep-logic-dependence} for a survey). The intuitive meaning of the atomic formulae are best fleshed out formally by the \textit{team semantics} capturing the (in)dependence between variables. The truth conditions of the logical connectives and other logical constants are also given based on teams, where one usual guideline is to define them in such a way that the language enjoys the property of \textit{flatness}, i.e., for any formula $\alpha$ without the (in)dependence atoms, it is true w.r.t.\ a team $X$  ($X\vDash \alpha$) if it is true on each singleton team $\{s\}$ such that $s\in X$. 
To some extent, flatness preserves the intuition of the classical logical connectives on possible worlds. In particular, the semantics of the distinct tensor disjunction $\tensor$ in dependence logic can be viewed as a natural lifting of the world-based semantics for classical disjunction to teams, viewed as \textit{sets} of possible worlds: 
$$ X\vDash \alpha\tensor\beta  \text{ iff there are $U,V\subseteq X$ such that $X\subseteq U\cup V$,} \text{ $U\vDash \alpha$ and $V\vDash\beta$}$$

Note that a disjunction $\alpha\lor\beta$ is classically true on each world in a set $X$ of possible worlds if and only if there are two subsets jointly covering the whole space of possible worlds such that one subset satisfies $\alpha$ homogeneously and the other satisfies $\beta$ homogeneously. This lifting may also give the impression that $\tensor$ can be read more or less as a classical disjunction. 
However, it is not so straightforward. For example, the truth of the propositional dependence formula $\dep(p,q)\ \tensor \dep(p,q)$ over a team is not equivalent to  $\dep(p,q)$. According to the semantics of $\tensor$, $\dep(p,q)\ \tensor \dep(p,q)$ says there are two subteams jointly covering the whole team, and $q$ depends on $p$ in each team. However, it is not necessarily that $q$ depends on $p$ over the whole team.
A natural question arises: how to understand this $\tensor$ disjunction intuitively and precisely?\footnote{In \cite{DLbook}, it is suggested that the  (in)dependence formulae can be viewed as \textit{types} of  teams.}
Our work proposes a possible epistemic understanding of $\tensor$ (and its generalizations) from a Brouwer-Heyting-Kolmogorov (BHK)-like perspective to be explained below.

The initial idea is based on an unexpected connection between the tensor disjunction and the so-called \textit{weak disjunction} in Medvedev's early work \cite{Med66} on the \textit{problem semantics} of intuitionistic logic, following Kolmogorov's problem-solving interpretation \cite{Kol32}. This connection is best exposed in the setting of \textit{inquisitive logic} with tensor disjunction discussed in  \cite{CiardelliB19}, since inquisitive logic has intimate connections with both the propositional dependence logic \cite{YangV16} and Medvedev's logic \cite{Ciardelli2011}. More specifically, various versions of propositional dependence logic can be viewed as the  disguised inquisitive logic, e.g., the dependence atom $\dep (p, q)$ becomes $(p\lor\neg p)\to (q\lor \neg q)$ \cite{Yangphd14,YangV16,Ciardelli2018}. On the other hand, Medvedev's logic is the substitution-closed core of inquisitive logic $\InqL$ that also admits a BHK-like interpretation via resolutions\cite{Ciardelli09,Ciardelli2011}.
\footnote{In the recent literature, inquisitive logic is also viewed as an extension of classical logic \cite{Ciardelli2016}.} 
Another advantage of using inquisitive logic as the ``medium'' is that we can put classical, intuitionistic, and tensor disjunctions in the same picture to reveal their differences. The last missing piece for an intuitive reading of tensor is an epistemic interpretation that can incorporate the BHK-interpretation. Wang proposed to capture intuitionistic truth  using a modality $\Kh$ to express \textit{knowing how to prove\slash solve} \cite{Wang2021}, which reflects Heyting's often-overlooked early view of intuitionistic logic as an epistemic logic \cite{Heyting56}. This also led to an \textit{alternative} epistemic  interpretation of inquisitive logic  \cite{Wang321}, where a state supports a formula $\alpha$ is rendered as \textit{it is known how to resolve $\alpha$} (more colloquially, \textit{knowing how $\alpha$ is true}) when viewing the state as a set of possible worlds capturing the epistemic  uncertainty. This can give us alternative epistemic readings of formulas in inquisitive logic. For example, $\neg \alpha$ in inquisitive logic is first rendered as $\Kh\neg \alpha$, which can be reduced to $\K\neg \alpha$ (\textit{knowing that} $\alpha$ does not have any resolution), reflecting the negation $\neg$ as the bridge between the intuitionistic and classical worlds. As another example, the excluded middle $\alpha\lor\neg \alpha$ in inquisitive logic is first rendered as $\Kh(\alpha\lor \neg \alpha)$, which is equivalent to $\Kh\alpha\lor\Kh\neg\alpha$ in our system, and eventually  can be reduced to the intuitively invalid  $\Kh \alpha\lor\K\neg \alpha$. When $\alpha$ is the atomic proposition $p$, $p\lor\neg p$ in inquisitive logic is equivalent to the epistemic formula $\K p\lor \K \neg p$ in our setting (see \cite{Wang321}). 

Now we are ready to give the epistemic interpretation of the tensor disjunction. According to Medvedev's problem semantics  \cite{Med66}, the weak disjunction $\alpha \sqcup \beta$ captures a composite problem where the solutions are pairs of \textit{potential} solutions to the problems of $\alpha$ and $\beta$ respectively such that \textit{at least one} solution in each pair is correct.\footnote{See \cite{CZbook}, for the corresponding Kripke semantics of weak disjunction. } From the epistemic interpretation, Medvedev's truth concept for a formula $\gamma$ means it is known how to solve $\gamma$. In particular, a weak disjunction $\alpha \sqcup \beta$ is true w.r.t.\ a set of possible worlds (i.e., a state\slash team) iff there are two solutions $r_1$ and $r_2$ such that it is known that one of $r_1$ and $r_2$ is a correct solution to the corresponding problems. We will show such a truth condition amounts to exactly the team semantics for the tensor. 

We first summarize what we actually did in the paper before going into the technical details. After introducing the inquisitive logic with tensor $\InqL^\tensor$ in Section \ref{sec.def}, we first propose in Section \ref{sec.deldef} a dynamic epistemic language of know-that and know-how, with extra machinery of announcements and propositional quantifiers, interpreted over epistemic models that are essentially states\slash teams in the literature. The semantics of the know-how operator is given based on a BHK-like interpretation, with the intention to capture the alternative epistemic meaning of $\InqL^\tensor$ formulae, which is formally justified by showing in Section \ref{sec.exp} that the valid know-how formulae are exactly theorems in $\InqL^\tensor$. Moreover, we also show that the announcements and propositional quantifiers facilitates a recursive process to ``open up'' the know-how formulae, in particular to decode the $\tensor$, and eventually translate them into classical ones free of the know-how operator. Based on such a process we give a complete axiomatization of our full dynamic epistemic logic in Section \ref{sec.axiom}. Finally, in Section \ref{sec.gtensor} we generalize the idea of the tensor, from our epistemic interpretation, to obtain a spectrum of $n$-ary disjunctions $\tensor^k_n$, which captures the interesting epistemic situation of knowing $n$ potential answers to $n$ questions and being sure at least $k$ of them must be correct. We show that adding the generalized tensor operators does not increase the expressive power of our logic, the inquisitive logic and propositional dependence logic, though most of these generalized tensors are not uniformly definable in these logics, except in our epistemic language.


\section{Preliminaries: Inquisitive Logic with Tensor Disjunction}\label{sec.def}
    Following \cite{CiardelliB19}, we introduce the language and semantics of  \textit{Inquisitive Logic with Tensor Disjunction} ($\InqL^{\tensor}$). In contrast with \cite{CiardelliB19}, we use the symbol $\lor$ for the inquisitive disjunction and adopt the model-based semantics as in \cite{Ciardelli2016}. Throughout the paper, we fix a countable set $\Prop$ of proposition letters.
    \begin{definition}[Language $\PL$]\label{def.LPP}
    The language of propositional logic with tensor ($\PL$) is defined as follows:
    $$\alpha::= p\mid \bot\mid (\alpha\land\alpha)\mid (\alpha\lor\alpha)\mid (\alpha\to\alpha)\mid(\alpha\tensor\alpha)$$
    where $p\in \Prop$. We write {$\neg\alpha$} for $\alpha\to\bot$, $\top$ and $\alpha\lra\beta$ are defined as usual.
    \end{definition}
     
    \begin{definition}[Model and state]\label{def.model}
   A model is a pair $\M=\langle W, \V\rangle$ where: 
        \begin{itemize}
            \item $W$ is a non-empty set of possible worlds;\footnote{In \cite{Ciardelli2020}, the world set $W$ could be empty. The distinction is not technically significant.}
            \item $\V: \Prop\to \wp(W)$ is a valuation function.
        \end{itemize}
        A \emph{state} $s$ in $\M$ is a subset of $W$.
    \end{definition}
    We will also view these models as \textbf{\textit{epistemic models}} for our dynamic epistemic language to be introduced in Section \ref{sec.deldef}. 

    Given $\M$, we refer to its components by $W_\M$ and $\V_\M$. We write $w\in \M$ in case that $w\in W_{\M}$, and $\M'\subseteq \M$ in case that $W_\M'\subseteq W_\M$.  The semantics is defined through the \textit{support relation} between \textit{states} (in models) and formulae.

\begin{definition}[Support \cite{CiardelliB19}]
\label{support}
The \emph{support} relation $\Vdash$ is defined inductively:

\begin{center}
    \begin{tabular}{|lcl|}
        \hline
        $\M,s\Vdash p$ &  iff & $\forall w\in s, w\in V(p)$ \\
		$\M,s\Vdash\bot$ &  iff&  $s=\varnothing$\\
		 $\M,s\Vdash (\alpha\land\beta)$ &  iff& $\M,s\Vdash\alpha$ and $\M,s\Vdash\beta$\\	
		$\M,s\Vdash (\alpha\lor\beta)$ &  iff & $\M,s\Vdash\alpha$ or $\M,s\Vdash\beta$\\
		$\M,s\Vdash (\alpha\to\beta)$ & iff& $\forall t\subseteq s:$ if $\M,t\Vdash\alpha$ then $\M,t\Vdash\beta$\\
		$\M,s\Vdash (\alpha\tensor\beta)$ &  iff& there exist two sets $t\subseteq s$ and $t'\subseteq s$ such that \\
		&& $\M,t\Vdash\alpha$,  $\M,t'\Vdash\beta$, and $t\cup t'=s$.\\
		\hline
        \end{tabular}
        \end{center}


A formula $\alpha$ is \emph{valid} if it is supported by any state in any model. 
\end{definition}
Here are some simple properties. 
\begin{proposition}[Downward closeness]\label{prop.downward}
For any $\alpha\in \PL$, if $\M,s \Vdash \alpha$ then $\M, t\Vdash \alpha$ for any $t\subseteq s$. Moreover, $\M,\emptyset \Vdash \alpha$ for all $\alpha\in \PL.$
\end{proposition}



\begin{definition} \textit{Inquisitive Logic with Tensor Disjunction} ($\InqL^{\tensor}$) is the set of valid $\PL$ formulae under the support relation. 


\end{definition}

\section{A dynamic epistemic language}\label{sec.deldef}
    \begin{definition}[Language $\LKhP$]
    The language of \emph{Public Announcement Logic with Know-how Operator and Propositional Quantifier} is defined as:\footnote{$\Pi$ in the name $\LKhP$ denotes  propositional quantifiers as in the literature  \cite{1970Fine}.}
    $$\phi::=  p \mid \bot \mid (\phi\land\phi)\mid (\phi\lor\phi)\mid(\phi\tensor\phi)\mid (\phi\to\phi)\mid\K\phi\mid \Kh\alpha \mid \forall p\phi \mid [\phi]\phi $$
    \noindent     where  $p\in \Prop$ and $\alpha\in\PL$. We write {$\hK$} for $\neg\K\neg$, {$\exists p$} for $\neg\forall p\neg$ for all $p\in\Prop$ and {$\langle \phi\rangle$} for $\neg[\phi]\neg$ for all $\phi\in\LKhP$.
    \end{definition}
Intuitively, $\K \phi$ expresses ``the agent \textit{knows that} $\phi$'', $\Kh\alpha$ says that ``the agent \textit{knows how} to resolve $\alpha$'' or simply ``the agent \textit{knows how}  $\alpha$ is true'', 
        $\forall p\phi$ says that ``for any  proposition $p$, $\phi$ holds'' and $[\phi]\psi$ means that ``after announcing $\phi$, $\psi$ holds''. 
        Note that $\Kh$ only allows $\PL$-formulae $\alpha$ in its scope
        . 
        For instance we can express $\K\neg \Kh\alpha$ but not $\Kh\K\alpha$ in $\LKhP$. 
We write $\phi[\psi\slash\chi]$ for any formula obtained by replacing one or several occurrences of $\psi$ with $\chi$ in $\phi$. 
    
    
    \medskip

    We view the models in Definition \ref{def.model} as \textit{epistemic models} where the implicit epistemic relation is the total relation. The semantics of $\LKhP$ is given on such models, with the notions of \textit{resolution space} and \textit{resolution} as below.
    \begin{definition}[Resolution space]\label{def.rs}
    $\S$ is a function assigning each $\alpha\in\PL$ its (non-empty) set of potential resolutions:\\
    
    \begin{minipage}{0.5\textwidth}
        $\begin{aligned}
        \S(p)&=\{p\},\text{ for }p\in\Prop\\
        \S(\alpha\lor\beta)&=(\S(\alpha)\times\{0\})\cup(\S(\beta)\times\{1\})\\
        \S(\alpha\land\beta)&=\S(\alpha)\times\S(\beta)\\
        \end{aligned}$
    \end{minipage} \quad\quad
    \begin{minipage}{0.45\textwidth}
       $ \begin{aligned}
        \S(\bot)&=\{\bot\}\\
        \S(\alpha\to\beta)&=\S(\beta)^{\S(\alpha)}\\
        \S(\alpha\tensor\beta)&=\S(\alpha)\times\S(\beta)\end{aligned}$
        \end{minipage}

    \end{definition}
    Resolution spaces reflect  the BHK-interpretation, e.g., a possible resolution of an implication is a function transforming a resolution of the antecedent into a resolution of the consequent. Note that resolution spaces for atomic propositions are singletons, based on the assumption in inquisitive semantics that atomic propositions are statements  without inquisitiveness. The set of actual resolutions of each formula on each world in a given model is a (possibly empty) subset of the corresponding resolution space, as  defined below. 

    \begin{definition}[Resolution in model]
    \label{def.resolution}
    Given $\M$, $\R\!:\! W_\M\!\times\!\PL\!\!\to\!\bigcup_{\alpha\in\PL}\S(\alpha)$ gives the (actual) resolutions for each $\PL$-formula on each world: 
        $$\begin{aligned}
        \R(w,\bot)&=\varnothing\qquad \R(w,p)=\left\{\begin{array}{cl}
    \{p\} & \text{ if $w\in V_\M(p)$}\\
      \emptyset &    \text{ otherwise} 
        \end{array}\right.\\
        \R(w,\alpha\lor\beta)&=(\R(w,\alpha)\times\{0\})\cup(\R(w,\beta)\times\{1\})\\
        \R(w,\alpha\land\beta)&=\R(w,\alpha)\times\R(w,\beta)\\
        \R(w,\alpha\to\beta)&=\{f\in\S(\beta)^{\S(\alpha)}\mid f[\R(w,\alpha)]\subseteq\R(w,\beta)\}\\
        \R(w,\alpha\tensor\beta)&=(\R(w,\alpha)\times\S(\beta))\cup(\S(\alpha)\times\R(w,\beta))\end{aligned}$$
    \paragraph{Important notation} For $U\subseteq W_\M$, we write $\R(U, \alpha)$ for $\bigcap_{w\in U}\R(w, \alpha).$
    \end{definition}

    While $\S(\bot)=\{\bot\}$ is non-empty, it never has any actual resolution on specific worlds. For any $p\in\Prop$, $p$ has itself as its resolution iff it is true on $w$. For any implication $\alpha\to \beta\in\PL$, each of its resolution on $w$ is a function in $\S(\alpha\to \beta)$ which maps an actual resolution of $\alpha$ to an actual  resolution of $\beta$ on $w$. Following the idea of the weak disjunction introduced in \cite{Med66}, each resolution for $\alpha\tensor\beta\in\PL$ on $w$ is a pair of resolutions in $\S(\alpha\tensor\beta)$, such that \textit{at least one} in the pair is actual on $w$ for the corresponding formula.
        
    Let $\Prop(\alpha)$ be the set of propositional letters occurring in $\alpha$ and let $V^\alpha_\M(w)$ be the collection of $p\in \Prop(\alpha)$ that are true on $w$ in $\M$.  
    Proposition \ref{prop.negation} is a useful observation on the resolution of negations ($\neg \alpha:= \alpha\to \bot$). Proposition \ref{prop.rind} says that $\R(w, \alpha)$ only depends on the relevant valuation on $w$ itself. 

\begin{proposition}[\citep{Wang321}]\label{prop.negation} 
For any $\M, w$, any $\alpha$,  
$\R(w, \neg \alpha)$ is either $\emptyset$ or a fixed singleton set independent from $w$, and $\R(w, \neg \alpha)=\emptyset$ iff $R(w, \alpha)\not=\emptyset$.
\end{proposition}
    
\begin{proposition}\label{prop.rind}
For any $\M,w$ and $\N,v$, for all $\alpha\in\PL$, if $V^\alpha_\M(w)=V^\alpha_\N(v)$, then $\R(w, \alpha)=\R(v,\alpha)$. 
\end{proposition}

Now we are ready to define the satisfaction relation of $\LKhP$ on \textit{pointed models}, i.e, a model with a designated world, in contrast with the state-based support-semantics. Note that the connectives outside the scope of $\Kh$ are classical, in particular $\tensor$ just functions as a classical disjunction. $\K$ is the standard epistemic modality of know-that. The semantics for $\Kh\alpha$ is defined via resolutions and is intended to capture the know-how interpretation of  $\InqL^{\tensor}$. $\forall p$ is a propositional quantifier over the full power set of $W_\M$. The semantics of the dynamic operator $[\psi]$ is as in public announcement logic \cite{2007Plaza}. 
    \begin{definition}[Semantics]\label{def.semantics}
    For $\phi,\psi\in\LKhP$, $\alpha\in\PL$ and $\M,w$ where $\M=\langle W,\V\rangle$, the satisfaction relation is defined as below where $\bigcirc\in \{\lor, \tensor\}$:\\
        $$\begin{array}{|lcl|}
        \hline
        \M,w\nvDash\bot & &  \\
        \M,w\vDash p &\iff& w\in \V(p)\\
        \M,w\vDash (\phi\bigcirc\psi)&\iff& \M,w\vDash \phi \text{ or }\M,w\vDash \psi\\
        \M,w\vDash (\phi\land\psi)&\iff& \M,w\vDash \phi \text{ and }\M,w\vDash \psi\\
        \M,w\vDash (\phi\to\psi)&\iff& \M,w\vDash \phi \text{ implies } \M,w\vDash \psi\\
        \M,w\vDash \K\phi&\iff& \text{ for any } v\in \M,  \M,v\vDash\phi\\
        \M,w\vDash \Kh\alpha&\iff& \text{ there exists an } x\in \S(\alpha) \text{ s.t. for any } v\in \M, x\in\R(v,\alpha)\\
        \M,w\vDash \forall p\phi&\iff& \text{ for any } U\in \wp(W_{\M}), \M[p\mapsto U],w\vDash\phi\\
        \M,w\vDash [\psi]\phi&\iff& \M,w\vDash\psi \text{ implies }  \M|_{\llbracket\psi\rrbracket},w\vDash\phi\\
        \hline
        \end{array}$$
        where:  
        \begin{itemize}
            \item Given $U\in \wp(W_\M)$ and $p\in\Prop$, recall that $\M[p\mapsto U]=\langle W, V'\rangle$, where the assignment $V'$ assigns $U$ to $p$ and coincides with $V$ on all other atoms; and
            \item $\llbracket\psi\rrbracket=\{w\in W_\M\mid\M,w\vDash\psi\}$ and $\M|_{X}$ is the submodel of $\M$ by restricting to $\emptyset\not=X\subseteq W_\M$. Thus $\M|_{\llbracket\psi\rrbracket}$ is the submodel restricted to the worlds satisfying $\psi$ in $\M$. We also write $\M|_{\llbracket\psi \rrbracket}$ as $\M|_\psi$ for brevity. 
        \end{itemize}
    \end{definition}
    \noindent Validity and entailment are defined as usual. 
    
    In \cite{Wang321}, we have a dynamic operator $\Box$. $\Box\phi$ says that ``given any information updates $\phi$ holds''. This can be expressed by $\forall p[p]\phi$ given that $p$ is not free in $\phi$, which is used to handle the implication in the know-how scope.



We write $\M\vDash \phi$ iff $\M,w\vDash \phi$ for all $w\in W_\M$. Apparently, $\M,w\vDash\Kh\alpha$ iff $\M\vDash \Kh\alpha$ and $\M,w\vDash\K\phi$ iff $\M\vDash\phi$. As mentioned in $\cite{Wang321}$, the semantics of $\Kh$ is in the $\exists x \K$ form as in other know-wh logics \cite{Wang2018,Wang17d}. The truth condition of $\Kh$ below says that $\Kh\alpha$ holds on a (pointed) model as long as there is a uniform resolution for $\alpha$ on that model, where we define $\R(U,\alpha)$ as $\bigcap_{w\in U}\R(w, \alpha).$

$$\begin{array}{|lclcl|}
\hline
\M\vDash \Kh\alpha&\iff&\M,w\vDash \Kh\alpha&\iff& \R(W_\M, \alpha)\not=\emptyset\\
\hline
\end{array}$$


\noindent An alternative truth condition for $\PL$-formulae can be given via resolutions.
    \begin{proposition}
    \label{prop.nept}
    For any $\alpha\in\PL$ and $\M,w$, $\M,w\vDash \alpha \iff \R(w,\alpha)\neq\varnothing$.
    \end{proposition}
        \begin{proof}
        We prove by induction on the structure of $\alpha$. We only show the cases for $\to$ and $\tensor$. The other cases can be found in \cite{Wang321}. \\
        $$\begin{aligned}
        \M,w\vDash (\alpha\to\beta) &\iff \M,w\vDash \alpha \text{ implies } \M,w\vDash \beta 
        \\&\iff \R(w,\alpha)\neq\varnothing \text{ implies } \R(w,\beta)\neq\varnothing\\&\iff\{f\in\S(\beta)^{\S(\alpha)}\}\neq\varnothing\text{ and } f[\R(w,\alpha)]\subseteq\R(w,\beta) \text{ is possible}\\&\iff\R(w,\alpha\to\beta)\neq\varnothing
         \end{aligned}$$
         $$\begin{aligned}
        \M,w\vDash (\alpha\tensor\beta) &\iff \M,w\vDash \alpha \text{ or } \M,w\vDash \beta \iff \R(w,\alpha)\neq\varnothing \text{ or }\R(w,\beta)\neq\varnothing\\&\iff \text{ there exists an } x\in\R(w,\alpha) \text{ or there exists a } y\in\R(w,\beta)\\&\iff\text{ there exists a pair } \langle x,x'\rangle \text{ or }\langle y',y\rangle \text{ in } \R(w,\alpha\tensor\beta)\\&\qquad\quad%
        \text{ such that }y'\in\S(\alpha) \text{ and } x'\in\S(\beta) \\&\iff\R(w,\alpha\tensor\beta)\neq\varnothing
        \end{aligned}$$
        \vspace{-8pt}
        \end{proof}
        
       From Proposition \ref{prop.nept} we see that in propositional formulae, both $\lor$ and $\tensor$  collapse to the  classical disjunction outside the scope of $\Kh$. Yet $\tensor$ is weaker than $\lor$  in the way that we can construct a resolution of $\alpha\tensor\beta$ from that of $\alpha\lor\beta$. It also follows from Proposition \ref{prop.nept} that for any $\alpha\in\PL$, $\M,w\vDash \K\phi$ iff for each $v\in\M$, there is some resolution for $\alpha$ on $v$. In contrast, $\M,w\vDash\Kh\alpha$ iff there is a \textit{uniform} resolution for $\alpha$ on $\M$. The following is immediate. 
    \begin{proposition}\label{prop.Kh2K}
    $\Kh \alpha\to \K \alpha$ is valid for all $\alpha\in\PL$. 
    \end{proposition}
    
        Since each $p\in\Prop$ only has one possible resolution, when each point has a resolution for $p$, the model  has a uniform one. Thus we have Proposition \ref{prop.K2Kh} 
        
    \begin{proposition}\label{prop.K2Kh}
     $\Kh p\lra \K p$ is valid for all $p\in\Prop$.
    \end{proposition}

    While the deduction rule \textit{replacement of equals by equals} is not valid in general, for instance, although $(p\lor\neg p)\lra (p\to p)$ is valid,  $\Kh(p\lor\neg p)\lra\Kh(p\to p)$ is not. However, if we only allow substitution to happen outside the scope of $\Kh$ operators, the rule becomes valid. It is not hard to verify the following:
    \begin{proposition}\label{prop.rRE}
    For $\phi,\psi,\chi\in\LKhP$, if $\phi\lra\psi$ is valid, then  $\chi[\phi\slash\psi]\lra \chi$ is valid, given that the substitution does not happen in the scope of $\Kh$. 
    \end{proposition}

\section{Expressivity}  \label{sec.exp}
    Let $\LPALPi$ be the $\Kh$-free fragment of $\LKhP$, $\LELPi$ be the $[\cdot]$-free fragment of $\LPALPi$ and $\LEL$ be the $\forall p$-free fragment of $\LELPi$. In Subsection \ref{sec.reduction}, we show $\Kh$ and $[\cdot]$ can be eliminated, thus making $\LKhP$,  $\LPALPi$ and $\LELPi$ equally expressive.
    In Subsection \ref{sec.=}, we show that the valid $\Kh$ formulae of $\LKhP$ corresponds to $\InqL^{\tensor}$ precisely.
\subsection{Reduction}\label{sec.reduction}
We introduce the reduction schemata to eliminate the $\Kh$ modality, which will also be used in the proof system to be introduced later. First, we have the following observation. 
    \begin{proposition}\label{prop.khtensor}
    For any $\alpha,\beta\in\PL$ where $p$ does not occur free, for any pointed model $\M,w$, $\M,w\vDash \exists p\K([p]\Kh\alpha\land[\neg p]\Kh\beta)$ iff there is a $U\subseteq W_\M,$ $(U\not=\emptyset$  implies $R(U,\alpha)\not=\emptyset)$  and $(\overline{U}\not=\emptyset$ implies $  R(\overline{U},\beta)\not=\emptyset)$. 
    \end{proposition}
    \begin{proof} Given a $U\subseteq W_\M$, for any $w\in\M$, $\M[p\mapsto U],w\vDash p\iff w\in U$ ($\star$). For brevity, we write $\exists U$ for there exists $U\in W_\M$. Recall that $\M|_U$ denotes the submodel of $\M$ restricted to $U$, if $U$ is non-empty (otherwise undefined).
      
      $$\begin{array}{ll}
          &\M,w\vDash \exists p\K([p]\Kh\alpha\land[\neg p]\Kh\beta) \\
         \iff & \exists U,  \M[p\mapsto U],w\vDash \K([p]\Kh\alpha\land[\neg p]\Kh\beta)\\
         \iff & \exists U, \forall v\in \M, \M[p\mapsto U],v\vDash [p]\Kh\alpha\land[\neg p]\Kh\beta\\
         & (\text{by $(\star)$ and the fact that $[\phi]\psi$ holds trivially if $\phi$ is false, we have: })\\
         \iff &\exists U, \forall v\in U, \M[p\mapsto U],v\vDash [p]\Kh\alpha \text{ and }\forall v\not\in U, \M[p\mapsto U],v\vDash [\neg p]\Kh\beta\\
          \iff &\exists U, \forall v\in U, \M[p\mapsto U]|_p,v\vDash \Kh\alpha \text{ and }\forall v\not\in U, \M[p\mapsto U]|_{\neg p},v\vDash \Kh\beta \\
          & (\text{since $p$ does not occur free in $\alpha$ and $\beta$, we have:})\\
        \iff &\exists U, \forall v\in U, \M|_U,v\vDash \Kh\alpha \text{ and }\forall v\not\in U, \M|_{\overline{U}},v\vDash [\neg p]\Kh\beta\\
        \iff &\exists U, (\text{$U\not=\emptyset$  implies }R(U,\alpha)\not=\emptyset) \text{ and } (\text{$\overline{U}\not=\emptyset$ implies } R(\overline{U},\beta)\not=\emptyset).\\
        \end{array}$$
    \end{proof}
    
    Together with Proposition \ref{prop.K2Kh} and \ref{prop.rRE}, Proposition  \ref{prop.khreduction} helps us to first eliminate the $\Kh$ modality without changing the expressive power, i.e., each $\LKhP$-formula is equivalent to a $\LPALPi$-formula. 
    \begin{proposition}\label{prop.khreduction}
    The following formulae and schemata are valid: 
        $$\begin{aligned}&\KKhp: &&  \K p\to\Kh p  \qquad\qquad\qquad \qquad \Khbot:  \quad   \Kh\bot\lra \bot \\
    	         &\KhD:  && \Kh(\alpha\lor\beta)\leftrightarrow \Kh\alpha\lor\Kh\beta \qquad \KhC:  \quad  \Kh(\alpha\land\beta)\leftrightarrow \Kh\alpha\land\Kh\beta \\
                 &\KhI:  &&  \Kh(\alpha\to\beta)\leftrightarrow \K\forall p [p](\Kh\alpha\to\Kh\beta) \text{, where } p \text{ does not occur free in }\alpha\text{ or }\beta\\
                 &\Kht:  &&  \Kh(\alpha\tensor\beta)\lra\exists p\K([p]\Kh\alpha\land[\neg p]\Kh\beta),\text{ where } p \text{ does not occur free in }\alpha\text{ or }\beta \\
    	\end{aligned}$$
    \end{proposition}
        \begin{proof}
        We only show the cases for $\KhI$ and $\Kht$. The rest of the proof can be found in \cite{Wang321}. 
\begin{itemize}
    \item [$\KhI$:]
    Recall that $\M,w\vDash \Box\phi \iff \text{ for any } \M'\subseteq \M \text{ s.t. $w\in \M'$},  \M',w\vDash\phi$. We claim that $\Box\phi$ can be defined by $\forall p\phi$ where $p$ does not occur free in $\phi$. Then it suffices to show that
 $\Kh(\alpha\to\beta)\lra \K\Box(\Kh\alpha\to\Kh\beta)$. The following proof comes from \cite{Wang321}.
 
    $\Longrightarrow$: Suppose $\M,w\vDash \Kh(\alpha\to\beta)$, then there is some $f\in \R(\M,\alpha\to\beta).$
 Towards a contradiction, suppose $\M,w\not\vDash\K\Box(\Kh\alpha\to\Kh\beta)$.
 That is, there is an $v\in \M$ and an $\M',v\subseteq \M,v$ s.t. $\M',v\vDash\Kh\alpha$ but $\M',v\not\vDash\Kh\beta$. So there is an $x\in \R(\M',\alpha)$. Recall that $f$ is a function with domain $\S(\alpha)$, and $\S(\alpha) \supseteq\R(u,\alpha)$ for all $u\in\M'$, thus $x\in Dom(f)$. Moreover, since $f\in \R(\M,\alpha\to\beta)$, $f\in \R(\M',\alpha\to\beta).$ Let $y=f(x)$. By the definition of $\R(\M',\alpha\to\beta)$, $y\in\R(u,\beta)$ for each $u\in \M'$. Therefore $\M',v\vDash\Kh\beta$, a contradiction.

        $\Longleftarrow$: Suppose $\M,w\vDash\K\Box(\Kh\alpha\to\Kh\beta)$, then for all $v\in\M$, $\M,v\vDash\Box(\Kh\alpha\to\Kh\beta)$. By the semantics of $\Box$, for any $v\in \M$ and for any $\M',v\subseteq \M,v$, $\M',v\vDash\Kh\alpha\to\Kh\beta$ ($\ast$). Since $\S(\alpha)$ is finite and non-empty, thus we can assume $\S(\alpha)=\{x_0,x_1,\dots,x_n\}$ for some $n\in\mathbb{N}$. For $i\in\{0,\dots,n\}$, let $W_i=\{w\mid x_i\in\R(w,\alpha)\}$. If $W_i$ is not empty then let $\M_i$ be the submodel of $\M$ such that $W_{\M_i}=W_i$. Clearly $x_i\in\R(W_i,\alpha)$, therefore for any $u\in \M_i$, $\M_i,u\vDash\Kh\alpha$. By ($\ast$) we have $\M_i,u\vDash\Kh\beta$ thus there is a $y_i\in\R(W_i,\beta)$. Now fix a $y\in \S(\beta)\not=\emptyset$, let $f=\{\langle x_i,y_i\rangle\mid i\in\{0,\dots,n\} \text{ and $W_i\not=\emptyset$}\}\cup \{\langle x_i,y\rangle\mid i\in\{0,\dots,n\} \text{ and $W_i=\emptyset$}\}$. Clearly $f\in\S(\alpha)^{\S(\beta)}$. Now for any $v\in \M$ and $i\in\{0,\dots,n\}$, if $x_i\in\R(w,\alpha)$ \text{ then } $v\in W_i$ by the definition of $W_i$, thus  $y_i\in\R(v,\beta)$ by the construction of $f$. Therefore $f[\R(v,\alpha)]\subseteq \R(v,\beta)$ for all $v\in \M$. It follows that $\M,v\vDash \Kh(\alpha\to\beta)$ for all $v\in \M$ including $w$. Note that the axiom of choice is not needed here.

 \item[$\Kht$:]        
        $\Longrightarrow$: Suppose $\M,w\vDash \Kh(\alpha\tensor\beta)$, then by the semantics, there is some $(x,y)\in \R(W_\M,\alpha\tensor\beta).$
        Let $U=\{v\in\M\mid x\in\R(v,\alpha)\}$. It is not hard to see $\overline{U}\subseteq \{v\in\M\mid y\in\R(v,\beta)\}$ by the definition of $\R(v,\alpha\tensor\beta)$. By Proposition \ref{prop.khtensor}, $\M,w\vDash \exists p\K([p]\Kh\alpha\land[\neg p]\Kh\beta).$  
        
        $\Longleftarrow$: Suppose $\M,w\vDash \exists p\K([p]\Kh\alpha\land[\neg p]\Kh\beta),$ by Proposition  \ref{prop.khtensor}, there is a $U$ satisfying the desired property. If $U\not=\emptyset$ and $\overline{U}\not=\emptyset$,  pick $(x, y)$ as the witness for $R(W_\M, \alpha\tensor\beta)$ such that $x\in R(U,\alpha)$ and $y\in R(\overline{U}, \beta)$. If $U=\emptyset$ then $\overline{U}\not=\emptyset$ since $W_\M$ is non-empty, and we pick $(x, y)$ such that $y\in R(\overline{U}, \beta)$ and $x\in S(\alpha)$. Similar for the case when $\overline{U}=\emptyset$. This suffices to show $\M,w\vDash \Kh(\alpha\tensor\beta).$
        \end{itemize}
    \end{proof}

        By Proposition \ref{prop.anreduction} we further eliminate the $[\cdot ]$ operator (without $\Kh$).\footnote{An alternative set of reduction formulae for the announcement operator is presented in Proposition 2.3 of \cite{2007Plaza} and Lemma 12 of \cite{16Hoek}.} 

    \begin{proposition}\label{prop.anreduction}
    The following formulae and schemata are valid: 
    $$\begin{aligned}  & \Anp   &&  {[ \chi]}p\lra(\chi\to p), \ p\in\Prop\cup\{\bot\}  \\
    			 & \Ancirc   &&  {[\chi]}(\phi\bigcirc\psi)\lra {[\chi]}\phi\bigcirc{[\chi]}\psi  ,   \bigcirc\in\{\land,\lor,\tensor,\to\}  \\
    			 & \AnK   &&  {[\chi]}\K\phi\lra (\chi\to\K({[\chi]}\phi))  \\
    	&\Anforall&& {[\chi]}\forall p\phi\lra\forall p{[\chi]}\phi, p \text{ is not in } \chi		 
    	\end{aligned}$$
    \end{proposition}
        \begin{proof}
        The only non-trivial case is $\Ancirc$ and we only show ${[\chi]}(\phi\lor\psi)\lra {[\chi]}\phi\lor{[\chi]}\psi$ as example.
        $\M,w\vDash[\chi](\phi\lor\psi)$ iff $\M,w\vDash\chi$ implies $\M|_{\chi},w\vDash \phi\lor\psi$ iff $\M,w\vDash\chi$ implies ($\M|_{\chi},w\vDash \phi$ or $\M|_{\chi},w\vDash \psi$) iff ($\M,w\vDash\chi$ implies $\M|_{\chi},w\vDash \phi$) or ($\M,w\vDash\chi$ implies $\M|_{\chi},w\vDash \psi$) iff $\M,w\vDash [\chi]\phi$ or $\M,w\vDash [\chi]\psi$) iff $\M,w\vDash{[\chi]}\phi\lor{[\chi]}\psi$.
        \end{proof}

    Without loss of generality, we can always rename the bound variable in case it occurs in $\chi$. Then for any $\Kh$-free formula $\phi$, by repeatedly applying Proposition \ref{prop.anreduction}, we can get rid of all $[\cdot]$ operators and find an equivalent $\LELPi$-formula for each $\LPALPi$-formula. We will give a formal presentation of this result in Theorem \ref{thm.express} as a natural consequence of Theorem \ref{thm.soundness} (Soundness).

\subsection{$\KhL=\InqL^{\tensor}$}\label{sec.=}
    
    Now we show that $\KhL=\{\alpha\in \PL \mid\ \vDash \Kh\alpha\}$ is exactly $\InqL^{\tensor}$.

    \begin{lemma}
    \label{lem.modelstate}
     For any $\alpha\in\PL$, $\M,w\vDash\Kh\alpha$ iff  $\M,W_\M\Vdash\alpha.$ As a consequence, for any non-empty state $s$ in $\M$, $\M,s\Vdash\alpha$ iff $\M|_s\vDash\Kh\alpha$. 
    \end{lemma}
        \begin{proof} Note that $\M,w\vDash\Kh\alpha$ iff $\M\vDash\Kh\alpha$ by the semantics, so we simply show $\M\vDash\Kh\alpha$ iff  $\M,W_\M\Vdash\alpha$ inductively on the structure of $\alpha$. We only prove the case for $\tensor$ and the rest are the same as in \cite{Wang321}. By Proposition \ref{prop.khreduction}, $\M\vDash \Kh(\alpha\tensor\beta)$ amounts to $\exists U,$ $(U\not=\emptyset$  implies $R(U,\alpha)\not=\emptyset)$  and $(\overline{U}\not=\emptyset$ implies $  R(\overline{U},\beta)\not=\emptyset)$. We show this is exactly $\M, W_\M\Vdash \alpha\tensor \beta.$ 
        
        $\Longrightarrow$: If both $U$ and $\overline{U}$ are  non-empty, then $\M\vDash \Kh(\alpha\tensor\beta)$ amounts to $\M|_U\vDash \Kh\alpha$ and  $\M|_{\overline{U}}\vDash \Kh\beta$. By IH, it is equivalent to $\M|_U, U\Vdash \alpha$ and  $\M|_{\overline{U}}, \overline{U}\Vdash \beta$, which implies $\M, W_\M\Vdash \alpha\tensor\beta$ since $U\cup \overline{U}=W_\M$. If one of $U$ and $\overline{U}$ is empty, suppose w.l.o.g. $U=\emptyset$, then we can also show $\M,\overline{U}\Vdash\beta$ (as before), and $\M,U\Vdash\alpha$, for the empty state support all formulae by Proposition \ref{prop.downward}. Thus $\M, W_\M\Vdash \alpha\tensor\beta$. 
        
        $\Longleftarrow$: Suppose $\M, W_\M\Vdash \alpha\tensor\beta$, then there are states $t$ and $t'$ such that $t\cup t'=W_\M$ and $\M,t\Vdash\alpha$ and $\M,t'\Vdash \beta$. Now at least one of $t$ and $t'$ is nonempty since $W_\M$ is non-empty. W.l.o.g., suppose $t\not=\emptyset$. Note that since $\overline{t}=(W_\M\setminus t) \subseteq t'$, then $\M,\overline{t}\Vdash \beta$ by Proposition \ref{prop.downward}. Now we take $U=t$, then by IH, $\M|_U\vDash \Kh\alpha$ and if $\overline{U}\not=\emptyset$ then $\M_{\overline{U}}\vDash\Kh \beta$. Therefore, $R(U,\alpha)\not=\emptyset$ and $(\overline{U}\not=\emptyset$ implies $  R(\overline{U},\beta \not=\emptyset)$. Thus, $\M, W_\M\Vdash\Kh(\alpha\tensor \beta)$ by Proposition \ref{prop.khtensor}. This concludes the first part of the proposition.
        
        For the consequence,  $\M|_s,w\vDash\Kh\alpha$ iff $\M|_s, s\Vdash \alpha$ iff $\M, s\Vdash\alpha$, and the last step is due to the fact the $\alpha$ only rely on the state in the support semantics.  
        \end{proof}
        

        
        
\begin{remark}\label{rem.equivsemantics}
Note that the proof for the $\tensor$ case above actually established the equivalence between our semantics based on the idea of weak disjunction by Medvedev and the team\slash support semantics in dependence\slash inquisitive logics. In our settings,  the formula $\dep (p,q)\ \tensor \dep(p,q)$ mentioned in the introduction says that there is a pair of dependence functions $(f_1,f_2)$ s.t. you know that one of these functions captures how $q$ depends on $p$.
\end{remark}

    Based on the lemma above, we can establish the relation between $\InqL^{\tensor}$ and $\KhL$, where $\Kh\Gamma=\{\Kh\alpha\mid \alpha\in \Gamma\}$.
    \begin{theorem}
    \label{ent}
     Given any $\{\alpha\}\cup \Gamma\subseteq \PL$, $\Gamma\Vdash \alpha$ iff $\Kh\Gamma\vDash \Kh \alpha$. As a consequence when $\Gamma=\emptyset$, $\InqL^{\tensor}=\KhL$.
    \end{theorem} 
    \begin{proof}
    Suppose $\Gamma\Vdash \alpha$ and $\M,w\vDash\Kh\Gamma$.  Now we have $\M,W_\M\Vdash \Gamma$ by Lemma \ref{lem.modelstate} thus $\M,W_\M\Vdash\alpha$, therefore $\M,w\vDash \alpha.$ For the other way around, if $\Kh\Gamma\Vdash \Kh\alpha$ and $\M,s\Vdash \Gamma$,  then $\M|_s\vDash \Kh\Gamma$ by Lemma \ref{lem.modelstate}, thus $\M|_s\vDash \Kh\alpha$. By Lemma \ref{lem.modelstate} again, $\M,s\Vdash\alpha.$   
    \end{proof}

\section{Axiomatization of $\LKhP$}\label{sec.axiom}
We first introduce the proof system $\SKhPALPi$ as below.\\
  \begin{center}
	{System $\SKhPALPi$}\\
\begin{tabular}{ll}
	\begin{tabular}{ll}
		{\textbf{Axioms}}&\\
			\TAUT & \text{Propositional tautologies}\\
			\tensoror &$(\phi\tensor\psi)\lra(\phi\lor\psi)$\\
			\DISTK &$\K (\phi\to\psi)\to (\K\phi\to \K\psi)$\\
		    \Anp &${[ \chi]}p\lra(\chi\to p), \ p\in\Prop\cup\{\bot\}$\\
			\Ancirc &${[\chi]}(\phi\!\bigcirc\!\psi)\lra {[\chi]}\phi\!\bigcirc\!{[\chi]}\psi$ \\
			\AnK &${[\chi]}\K\phi\lra \chi\to\K[\chi]\phi$\\
		    \Anforall & ${[\chi]}\forall p\phi\lra\forall p{[\chi]}\phi$, $p$ is not in $\chi$\\ 
			\DISTA & $\forall p(\phi\to\psi)\to(\forall p\phi\to\forall p\psi)$\\
			\subforall & $\forall p\phi\to\phi[\psi\slash p]$, $\psi$ is free for $p$ in $\phi$\\%
		   	\SU & $\exists p(p\land\forall q(q\to\K(p\to q)))$\\
            \Barcan & $\forall p \K\phi\to\K\forall p\phi$\\
             $\KhK$ & $\Kh\alpha\to\K\alpha$\\
	        $\KKhp$& $\K p\to\Kh p$\\
	         $\Khbot$ & $\Kh\bot\lra \bot$\\
            $\KhD$ & $\Kh(\alpha\lor\beta)\leftrightarrow \Kh\alpha\lor\Kh\beta$\\
            $\KhC$ & $\Kh(\alpha\land\beta)\leftrightarrow \Kh\alpha\land\Kh\beta$\\
            $\KhI$ & $\Kh(\alpha\to\beta)\leftrightarrow \K\forall p [p](\Kh\alpha\to\Kh\beta)$\\
            $\Kht$ & $\Kh(\alpha\tensor\beta)\lra\exists p\K([p]\Kh\alpha\land[\neg p]\Kh\beta)$\\
	\end{tabular} &
	\begin{tabular}{ll}
	    \AxTK&$\K \phi\to \phi$ \\
		\AxTransK&$\K \phi\to\K\K \phi$\\
		\AxEucK&$\neg \K \phi\to\K\neg\K \phi$\\
		$\AxTransKh$ & $\Kh\alpha\to \K\Kh\alpha$\\
	    $\AxEucKh$& $\neg\Kh\alpha\to \K\neg\Kh\alpha$\\
	    \ \\
	\textbf{Rules}&\\
    	\MP & $\dfrac{\phi,\phi\to\psi}{\psi}$\\
    	\NECK& $\dfrac{\vdash\phi}{\vdash\K\phi}$ \\
    	\generalization &$\dfrac{\vdash\phi\to\psi}{\vdash \phi\to\forall p \psi}$ \\ 
    	&  $p$ not free in $\phi$\\
    	\rRE& $\dfrac{\vdash\phi\lra\psi}{\vdash \chi[\phi\slash\psi]\lra \chi}$, \\ 
    	& given that the \\
    	& substitution \\
    	& does not happen \\
    	& in the scope of $\Kh$	
    	\end{tabular}\\
\end{tabular}\\
	\end{center}
\noindent \textit{where $p\in\Prop$, $\alpha,\beta\in \PL$,  $\phi,\psi,\chi\in\LKhP$, $\bigcirc\!\in\!\{\land,\!\lor,\!\tensor,\!\to\}$; $p$ does not occur free in $\alpha$ and $\beta$ in $\KhI$ and $\Kht$.}

    Together with $\rRE$, $\tensoror$ states the fact that $\tensor$ behaves exactly like $\lor$ when it  occurs outside $\Kh$. $\SFive$ axiom schemeta\slash rules for $\K$ together with $\TAUT$, $\DISTA$, $\subforall$, $\SU$ and rule $\generalization$ form a complete axiomatization $\SFivePiplus$ of S5 logic with propositional quantifiers \citep{1970Fine}, where $\SU$ states the existence of atoms. Operators $\Anp$, $\Ancirc$,  $\AnK$  and $\Anforall$ are reduction axioms for $[\cdot]$ \cite{2007Plaza,16Hoek}.\footnote{The original form of $\Anforall$ in \cite{16Hoek} is ${[\chi]}\forall p\phi\lra (\chi\to\forall p{[\chi]}\phi)$ ($p$ is not in $\chi$).} $\KKhp$, $\Khbot$, $\KhD$, $\KhC$, $\KhI$ and $\Kht$ are the reduction axioms decoding the $\PL$ formulae, whose usages are shown in Lemma $\ref{lem.redkh}$. Barcan Formula $\Barcan$,  introspection schemata $\AxTransK$, $\AxTransKh$ and $\AxEucKh$ can be proved from the rest of the system. In particular, $\AxTransKh$ requires an inductive proof on the structure of $\alpha$. We include them for their intuitive meanings. 
    
    In order to show the power of $\SKhPALPi$, we give some examples of provable formulae in the system.
    \begin{proposition}
The following are provable in $\SKhPALPi$: 
\begin{center}
\begin{tabular}{ll}
$\Anid$&${[\chi]}\phi\lra(\chi\to{[\chi]}\phi)$\\
$\Anexists$ & ${[\chi]}\exists p\phi\lra\exists p{[\chi]}\phi$, $p$ is not in $\chi$\\
\end{tabular}
\end{center}
\end{proposition}
\counterwithout{equation}{section}
\setcounter{equation}{0}
\begin{proof}
    For $\Anid$:
    Following Lemma $\ref{lem.redkh}$, We first change each $\LKhP$-formula $\phi$ into the $\LPALPi$-formula $\phi'$ such that $\phi'$ is provably equivalent to $\phi$. With Rule $\rRE$, we only need to construct the proof of ${[\chi]}\phi'\lra(\chi\to{[\chi]}\phi')$.
    
    We prove by induction on $\phi'$ to show that there is always a proof for $\Anid$ in $\SKhPALPi$.
    
    \begin{itemize}
        \item  If $\phi'\in\Prop\cup\{\bot\}$, then we construct the following proof.
            \begin{align}
                &\vdash{[\chi]}\phi'\lra(\chi\to\phi') &&\Anp\label{001}\\
                &\vdash(\chi\to\phi')\lra(\chi\to(\chi\to\phi')) &&\TAUT\label{002}\\
                &\vdash{[\chi]}\phi'\lra(\chi\to{[\chi]}\phi') &&(\ref{001})(\ref{002})\rRE
            \end{align} 
        \item If $\phi'$ is $\phi_1\bigcirc\phi_2$, $\bigcirc\in\{\land,\lor,\tensor,\to\}$, we construct the following proof.
\setcounter{equation}{0}
            \begin{align}
                &\vdash{[\chi]}(\phi_1\bigcirc\phi_2)\lra {[\chi]}\phi_1\bigcirc{[\chi]}\phi_2
                &&\Ancirc\label{011}\\
            	&\vdash{[\chi]}(\phi_1\bigcirc\phi_2)\lra (\chi\to{[\chi]}\phi_1)\bigcirc(\chi\to{[\chi]}\phi_2) &&(\ref{011})\rRE,\text{IH}\label{012}\\
            	&\vdash{[\chi]}(\phi_1\bigcirc\phi_2)\lra \chi\to({[\chi]}\phi_1\bigcirc{[\chi]}\phi_2)
            	&&(\ref{012})\TAUT\label{013}\\
            	&\vdash{[\chi]}(\phi_1\bigcirc\phi_2)\lra \chi\to{[\chi]}(\phi_1\bigcirc\phi_2)
            	&&(\ref{013})(\ref{011})\rRE
            \end{align} 
        \item If $\phi'$ is $\K\psi$, we construct the following proof.
\setcounter{equation}{0}
            \begin{align}
                &\vdash{[\chi]}\K\psi\lra (\chi\to\K{[\chi]}\psi)
                &&\AnK\label{021}\\
                &\vdash(\chi\to\K{[\chi]}\psi)\lra (\chi\to(\chi\to\K{[\chi]}\psi)) &&\TAUT\label{022}\\
                &\vdash{[\chi]}\K\psi\lra (\chi\to(\chi\to\K{[\chi]}\psi))
                &&(\ref{021})(\ref{022})\rRE\label{023}\\
                &\vdash{[\chi]}\K\psi\lra (\chi\to{[\chi]}\K\psi)
            	&&(\ref{023})\rRE
            \end{align} 
        \item If $\phi'$ is $\forall p\psi$,  construct the following proof. Let $q\in\Prop$ be the first propositional variable that is not in $[\chi]\phi$.
\setcounter{equation}{0}
            \begin{align}
                &\vdash\forall p\psi\lra\forall q\psi[q\slash p]
                &&\subforall,\generalization\label{031}\\
                &\vdash{[\chi]}\forall q\psi\lra \forall q{[\chi]}\psi
                &&\Anforall\label{032}\\
                &\vdash{[\chi]}\forall q\psi\lra \forall q(\chi\to{[\chi]}\psi) &&(\ref{032})\rRE,\text{IH}\label{033}\\
                &\vdash{[\chi]}\forall q\psi\lra (\forall q\chi\to\forall q{[\chi]}\psi) &&(\ref{033})\TAUT,\DISTA,\subforall,\generalization\label{034}\\
                &\vdash\chi\lra\forall q\chi\  (q\text{ is not in } \chi)
    	        &&\subforall,\generalization\label{035}\\
                &\vdash{[\chi]}\forall q\psi\lra (\chi\to\forall q{[\chi]}\psi)
                &&(\ref{034})(\ref{035})\rRE\label{036}\\
                &\vdash{[\chi]}\forall q\psi\lra (\chi\to{[\chi]}\forall q\psi)
                &&(\ref{036})\rRE\Anforall\label{037}\\
                &\vdash{[\chi]}\forall p\psi\lra (\chi\to{[\chi]}\forall p\psi)
                &&(\ref{037})(\ref{031})\rRE
            \end{align} 
    \end{itemize}
    
    For $\Anexists$: By definition of $\exists p$ and $\neg$ we only have to prove  ${[\chi]}(\forall p(\phi\to\bot)\to\bot)\lra\forall p({[\chi]}\phi\to\bot)\to\bot$, where $p$ is not in $\chi$.
\setcounter{equation}{0}    
    \begin{align}
    	&\vdash{[\chi]}(\forall p(\phi\to\bot)\to\bot)\lra{[\chi]}\forall p(\phi\to\bot)\to{[\chi]}\bot &&\Ancirc\label{106}\\
    	&\vdash{[\chi]}(\forall p(\phi\!\to\!\bot)\!\to\!\bot)\lra\forall p{[\chi]}(\phi\!\to\!\bot)\!\to\!(\chi\!\to\!\bot) &&(\ref{106})\rRE,\Anp,\Anforall\label{107}\\
    	&\vdash{[\chi]}(\forall p(\phi\to\bot)\to\bot)\lra(\forall p{[\chi]}(\phi\to\bot)\land\chi)\to\bot &&(\ref{107})\TAUT\label{108}\\
    	&\vdash\chi\lra\forall p\chi, 
    	&&\subforall,\generalization\label{119}\\
    	&\vdash{[\chi]}(\forall p(\phi\to\bot)\to\bot)\lra(\forall p{[\chi]}(\phi\to\bot)\land\forall p\chi)\to\bot &&(\ref{108})(\ref{119})\rRE\label{109}\\
    	&\vdash{[\chi]}(\forall p(\phi\to\bot)\to\bot)\lra\forall p{[\chi]}((\phi\to\bot)\land\chi)\to\bot &&(\ref{109})\TAUT,\DISTA,\subforall,\generalization\label{110}\\
    	&\vdash{[\chi]}(\forall p(\phi\to\bot)\to\bot)\lra\forall p(({[\chi]}\phi\to{[\chi]}\bot)\land\chi)\to\bot
    	&&(\ref{110})\rRE,\Ancirc\label{111}\\
    	&\vdash{[\chi]}(\forall p(\phi\to\bot)\to\bot)\lra\forall p(({[\chi]}\phi\!\to]!(\chi\!\to\!\bot))\!\land\!\chi)\!\to\!\bot
    	&&(\ref{111})\rRE,\Anp\label{112}\\
    	&\vdash{[\chi]}(\forall p(\phi\to\bot)\to\bot)\lra\forall p(({[\chi]}\phi\to\bot)\land\chi)\to\bot
    	&&(\ref{112})\TAUT\label{113}\\
    	&\vdash{[\chi]}\phi\to\bot\lra({[\chi]}\phi\to\bot)\land\chi &&\Anid,\TAUT\label{114}\\
    	&\vdash{[\chi]}(\forall p(\phi\to\bot)\to\bot)\lra\forall p({[\chi]}\phi\to\bot)\to\bot &&(\ref{113})(\ref{114})\rRE
    \end{align}
\end{proof}

\subsection{Provable equivalence}


    In Section \ref{sec.reduction}, we showed $\LKhP$ is expressively equivalent to $\LELPi$. Now we show the same result by referring to the soundness of $\SKhPALPi$ (Theorem \ref{thm.soundness}) and that each $\LKhP$-formula $\varphi$ is \textit{provably equivalent} to a $\LELPi$-fromula $\varphi'$ (Lemma \ref{lem.express}). Meanwhile we provide a  translation from $\varphi$ to  $\varphi'$.

    \begin{theorem}[Soundness]\label{thm.soundness}
    $\SKhPALPi$ is sound over the class of all  models. 
    \end{theorem}
        \begin{proof}
        The validity of $\Anp$, $\Ancirc$, $\AnK$ and $\Anforall$ are given in Proposition  \ref{prop.anreduction}. $\DISTA$, $\subforall$, $\SU$ and rule $\generalization$ are given in \cite{1970Fine}. $\KKhp$, $\Khbot$, $\KhD$, $\KhC$, $\KhI$, and $\Kht$ are given in Proposition \ref{prop.khreduction}. $\rRE$ is given in Proposition \ref{prop.rRE}. The rest are trivial. \end{proof}

    To prove the completeness we first prove Lemmata \ref{lem.redkh} and \ref{lem.redan} with the two sets of reduction axioms.
    Recall that $\LPALPi$ is the $\Kh$-free fragment of $\LKhP$, and $\LELPi$ is the $[\phi ]$-free fragment of $\LPALPi$. 

   \begin{lemma} \label{lem.redkh}
    Each $\LKhP$-formula is provably equivalent to a $
    \Kh$-free $\LPALPi$ formula in $\SKhPALPi$. 
    \end{lemma}
        \begin{proof}
        We use $\rRE$ and Axioms $\Khbot, \KhC, \KhD, \KhI$, $\Kht$ repeatedly to reduce $\Kh\alpha$ to some formula with $\Kh p$ only. With $\vdash \Kh p\lra \K p$ from $\KhK$ and $\KKhp$, we can eliminate all $\Kh$ modalities. 
        \end{proof}

To eliminate the announcement operator, we need a notion of complexity. 
    \begin{definition}[Announcement rank]
    For each $\phi\in\LPALPi$, we define its \textit{announcement rank} $\arank(\phi)$ inductively as follows:
    \begin{itemize}
        \item If $\phi=p$ or $\phi=\bot$, then $\arank(\phi)=0$.
        \item If $\phi=\psi_1\bigcirc\psi_2$ where $\bigcirc=\land,\lor,\tensor$ or $\to$
        , then $\arank(\phi)=\max\{\arank(\psi_1),\arank(\psi_2)\}$.
        \item If $\phi=\K\psi$, $p\in\Prop$
        , then $\arank(\phi)=\arank(\psi)$.
        \item If $\phi=\forall p\psi$, $p\in\Prop$
        , then $\arank(\phi)=\arank(\psi)$.
        \item If $\phi=[\chi]\psi$
        , then $\arank(\phi)=\arank(\psi)+\arank(\chi)+1$.
    \end{itemize}
    \end{definition}

    \begin{lemma}\label{lem.anrank}
    Each $\LPALPi$-formula of the form $[\chi]\psi$ is provably equivalent to a  $\LPALPi$-formula $\phi$ in $\SKhPALPi$ such that  $\arank(\phi)<\arank([\chi]\psi)$. 
    \end{lemma}
       \begin{proof} We prove by induction on $n=\arank([\chi]\psi)$. By definition, $n\geq 1$. In the induction base, suppose $n=1$, then $\arank(\chi)=\arank(\psi)=0$. We prove by induction on $\psi$ that there is a $\phi$ such that $\phi\lra[\chi]\psi$ and $\arank(\phi)<n$.
            \begin{enumerate}
            \item If $\psi=p$ or $\psi=\bot$, then by axiom $\Anp$, $[\chi]\psi\lra\chi\to\psi$. Hence $\phi=\chi\to\psi$ is what we need. 
            \item If $\psi=\psi_1\bigcirc\psi_2$ where $\bigcirc=\land,\lor,\tensor,\to$
            , then by $\Ancirc$, $[\chi]\psi\lra[\chi]\psi_1\bigcirc[\chi]\psi_2$. By IH, there are $\phi_1\lra[\chi]\psi_1$ and $\phi_2\lra[\chi]\psi_2$ such that $\arank(\phi_1)<\arank([\chi]\psi_1)$ and $\arank(\phi_1)<\arank([\chi]\psi_1)$. $\phi=\phi_1\bigcirc\phi_2$ is what we need. 
            \item If $\psi=\K\psi'$, then by $\AnK$, ${[\chi]}(\psi)\lra\chi\to\K{[\chi]}\psi'$. Note that $\arank(\chi)<\arank([\chi]\psi)$ and $\arank([\chi]\psi)=\arank(\K{[\chi]}\psi')$ by definition. By IH, we find $\phi'\lra [\chi]\psi'$. $\phi=\chi\to\K\phi'$ is what we need.
            \item If $\psi=\forall p\psi'$ where $p\in\Prop$, we consider two subcases.1).if $p$ is not in $\chi$, we use $\Anforall$ and the proof is similar to the above cases. 2).if $p$ is in $\chi$,  replace $p$ with the first letter $q\in\Prop$ which is not in $\chi$ and then go to 1).
            \end{enumerate}
        In the induction step, suppose $n>1$. Since $\arank([\chi]\psi)=\arank(\chi)+\arank(\psi)+1$, either $1\leq \arank(\chi)\leq n$ or $1\leq \arank(\psi)\leq n$. Assume that $1\leq \arank(\chi)\leq n$. By IH, we find a $\chi'\lra\chi$ s.t. $\arank(\chi')<\arank(\chi)$. And $\phi=[\chi']\psi$ has the desired properties. The other case is similar.
        \end{proof}
     The idea is that we start from the innermost subformulae, and replace them with equivalent $\LELPi$-formulae using the reduction axioms and $\rRE$. In this way, we can always get an equivalent formula with lower announcement rank. Since the  announcement rank is finite, we can decrease the rank till zero eventually by repeating the process above. Therefore we have the following Lemma \ref{lem.redan} 
    \begin{lemma}\label{lem.redan}
    Each $\LPALPi$-formula is provably equivalent to an $\LELPi$-formula in $\SKhPALPi$. 
    \end{lemma}

    Combining Lemmata \ref{lem.redkh} and \ref{lem.redan} we immediately have.
    \begin{lemma}\label{lem.express}
     Each $\LKhP$-formula is provably equivalent to an $\LELPi$-formula in $\SKhPALPi$. 
    \end{lemma}

    Theorem \ref{thm.express} follows naturally from Lemma \ref{lem.express} and Theorem \ref{thm.soundness}.
    \begin{theorem}\label{thm.express}
    $\LKhP$ is equally expressive as $\LELPi$ over all  models. 
    \end{theorem}
Note that $\LELPi$ is more expressive than $\LEL$ \cite{1970Fine}.
\subsection{Completeness}
    With Lemma \ref{lem.express} and Theorem \ref{thm.express}, the completeness of System $\SKhPALPi$ can be reduced to that of $\SFivePiplus$, which is given in \cite{1970Fine}. $\SFivePiplus$ is a variety of second order modal logic, containing all the axiom schmeta\slash rules of $S5$ as well as those concerning propositional quantifiers in $\SKhPALPi$.

\begin{theorem}[Completeness of $\SFivePiplus$ \citep{1970Fine}]\label{prop.s5+}
$\SFivePiplus$ is a complete axiomatization with regard to the class of models.
\end{theorem}
\begin{theorem}[Completeness]\label{thm.comppalkhpi}
 System $\SKhPALPi$ is a complete axiomatization of $\InqKhL$.
\end{theorem}
\begin{proof}
We first use Lemma \ref{lem.redkh} and Lemma \ref{lem.redan} to translate each $\LKhP$-formula $\phi$ into an equivalent $\LELPi$-formula $\phi'$ and then use the completeness of $\SFivePiplus$. Note that $\vdash\phi$ below means $\phi$ is in $\SKhPALPi$. 
  $$\vDash\phi\  
  \underset{\text{Theorem  \ref{thm.express}}}{\overset{\text{expressive equivalence}}{\iff}} 
  \ \vDash \phi'\ 
  \underset{\text{Theorem  \ref{prop.s5+}}}{\overset{\text{completeness of } \SFivePiplus}{\iff}}
  \ \vdash_{\SFivePiplus}\phi'\\$$
  $$   
  \overset{\SFivePiplus\subseteq\SKhPALPi}{\implies}
  \ \vdash\phi'\ 
  \underset{\text{Lemma  \ref{lem.express}}}{\overset{\text{\text{provable equivalence}}}{\iff}}
  \ \vdash\phi$$
\end{proof}
\section{Generalization of Tensor Disjunction}\label{sec.gtensor}
Inspired by our epistemic interpretation, we generalize the binary $\tensor$ to $n$-ary operators for any $n\geq 2$ with another parameter $k\leq n$. 
\subsection{Generalizing the tensor operator} 
    Consider the following scenario: You completed an exam with $n$ questions with one point each, and get a total score of $m$ without knowing which of your answers were correct. What is your epistemic state? The original tensor  actually captures the special case when $m=1$ and $n=2$: you have two resolutions for $\alpha$ and $\beta$ respectively, and you are sure at least one of them must be an actual resolution for the corresponding formula. For any $n\geq 2$ and $1\leq m\leq n$, we now define an $n$-ary connective $\tensor^k_n$. 


\begin{definition}[Language $\PLT$]
The \emph{propositional language with general tensor} ($\PLT$) is as follows:
$$\alpha::= p\mid \bot\mid (\alpha\land\alpha)\mid (\alpha\lor\alpha)\mid (\alpha\to\alpha)\mid \tensor_n^k(\underbrace{\alpha,\cdots,\alpha}_{n})
\vspace{-7pt}
$$
\noindent where $p\in \Prop$ and $n\geq 2$, $1\leq k\leq n$. 
\end{definition}

\begin{definition}[Language $\LKhPT$] The \emph{Public Announcement Logic with Know-how and General Tensor} ($\LKhPT$) is as follows:
$$\phi::= p\mid \bot \mid (\phi\land\phi)\mid (\phi\lor\phi)\mid (\phi\to\phi)\mid \tensor_n^k(\underbrace{\phi,\cdots,\phi}_{n})\mid\K\phi\mid \Kh\alpha \mid \forall p\phi \mid [\phi]\phi 
\vspace{-10pt}
$$
\noindent where  $p\in \Prop$ and $\alpha\in\PLT$. 
\end{definition}

Now, we introduce the semantics of new connectives $\tensor_n^k$ via resolutions.
\begin{definition}\label{def.general R}

For any positive integer $n\geq 2$ and $1\leq k\leq n$, we define the resolution space and resolution of $\tensor_n^k$ as follow:
$$\begin{aligned}
\S(\tensor_n^k({\alpha_1,\cdots,\alpha_n}))&=\S(\alpha_1)\times\cdots\times\S(\alpha_n)\\
\R(w,\tensor_n^k({\alpha_1,\cdots,\alpha_n}))&=\{(r_1,\cdots,r_n) \mid k\leq | \{ i\in [1,n] \mid r_i\in \R(w,\alpha_i) \}| \}
\end{aligned}$$
\end{definition}

The truth condition for $\Kh$ is as before in Definition \ref{def.semantics}. In particular, $\M,w\vDash \Kh\tensor^k_n(\alpha_1,\cdots,\alpha_n)$ iff $\R(W_\M, \tensor^k_n(\alpha_1,\cdots,\alpha_n))\neq\emptyset$.

By Definition \ref{def.semantics} and \ref{def.general R}, it is not hard to see the following. 
\begin{proposition}\label{prop.KhTeq}
$\M,w\vDash\Kh\tensor_n^k({\alpha_1,\cdots,\alpha_n})$ if and only if there is an n-tuple $(r_1,\cdots,r_n)$ such that for any $v\in W_\M$, $|\{i\mid r_i\in\R(v,\alpha_i)\}|\geq k,$ i.e., there are at least $k$ indexes $i\in [1, n]$ such that $r_i\in\R(v,\alpha_i)$. 
\end{proposition}


Note that based on the above proposition, the truth condition for $\tensor^1_2$ is exactly as the one for the standard $\tensor$ defined earlier.  


$\tensor^k_n$ can also appear out of $\Kh$. Hence we define its semantics as below. 

\begin{definition}[Semantics]\label{def.gt}
$$\begin{aligned}
\M,w\vDash \tensor_n^k({\phi_1,\cdots,\phi_n}) & \iff \M,w\vDash \bigvee_{\substack{I\subseteq\{1,2,\cdots,n\}\\ | I |=k}}\bigwedge_{i\in I}\phi_i\\
\end{aligned}$$

\end{definition}
\noindent The semantics is guided by Proposition \ref{prop.nept}, with the desired property below. 
\begin{proposition}
For any $\alpha\in\PLT$ and $\M,w$, $\M,w\vDash \alpha \iff \R(w,\alpha)\neq\varnothing$.
\end{proposition}

 \begin{proof}
 Based on Proposition \ref{prop.nept}, we only consider the case of $\tensor_n^k({\alpha_1,\cdots,\alpha_n})$. 
 $$\begin{aligned}
 &\M,w\vDash \tensor_n^k({\alpha_1,\cdots,\alpha_n})\\
 \iff & \M,w\vDash \bigvee_{\substack{I\subseteq\{1,2,\cdots,n\}\\ | I |=k}}\bigwedge_{i\in I}\alpha_i\\
 \iff & \exists I\subseteq\{1,2,\cdots,n\} \text{ with } | I |=k \text{ s.t. } \M,w\vDash\bigwedge_{i\in I}\alpha_i\\
 \iff & \exists I\subseteq\{1,2,\cdots,n\} \text{ with } | I |=k \text{ s.t. } \forall i\in I, \R(w,\alpha_i)\neq\emptyset \text{(by IH)} \quad (\dagger)\\
 \end{aligned}$$

 And it is easy to see that $\R(w,\tensor_n^k({\alpha_1,\cdots,\alpha_n}))$ is nonempty iff at least $k$ of $\R(w,\alpha_i)$ is nonempty. Hence, $(\dagger)$ implies that $\R(w,\tensor_n^k({\phi_1,\cdots,\phi_n}))\neq\emptyset$. 
 \end{proof}

Next, we show how to reduce the general tensors in $\LKhPT$.
\begin{proposition}\label{prop.semantic tensor}
The following schemata are valid: 
$$\begin{aligned}
&\gt &&  \tensor_n^k({\phi_1,\cdots,\phi_n})\lra\bigvee_{\substack{I\subseteq\{1,2,\cdots,n\}\\ | I |=k}}\bigwedge_{i\in I}\phi_i \\
\vspace{-10pt}
&\Khgt  &&  \Kh\tensor_n^k({\alpha_1,\cdots,\alpha_n})\lra\exists p_1\cdots\exists p_n(\K\tensor_n^k(p_1,\cdots,p_n)\land\bigwedge_{i=1}^{n}\K[p_i]\Kh\alpha_i) \\
\end{aligned}
\vspace{-5pt}
$$
(where all the $p_i$ do not occur free in all the $\alpha_i$)
\end{proposition}
\begin{proof}
$\gt$ is valid by the truth condition of $\tensor^k_n$ in Definition \ref{def.gt}.

For $\Khgt$: 
\begin{itemize}
    \item[$\Longrightarrow$] By Proposition \ref{prop.KhTeq} $\M,w\vDash \Kh\tensor_n^k({\alpha_1,\cdots,\alpha_n})$ iff there is an n-tuple $(r_1,\cdots,r_n)$ s.t. for any $v\in W_\M$, there are at least $k$ indexes $i\in [1,n]$ such that $r_i\in\R(v,\alpha_i)$. 
    Let $U_i=\{v\in W_\M \mid r_i\in \R(v,\alpha_i)\}$, then consider $\M[\bar{p}\mapsto\bar{U}] =\langle W,\V'\rangle$ such that $V'$ assigns $U_i$ to $p_i$ for $i\in \{1,\dots,n\}$ and coincides with $V$ on all other atoms. 
    Then, for any $v\in W_\M$, there are at least $k$ indexes $i\in [1,n]$ s.t. $\M[\bar{p}\mapsto\bar{U}],v\vDash p_i$, so $\M[\bar{p}\mapsto\bar{U}],w\vDash \K\tensor_n^k(p_1,\cdots,p_n)$. 
    And since for any $v\in U_i$ we have $r_i\in\R(v,\alpha_i)$, so we have for any $v\in W_\M$, $\M[\bar{p}\mapsto\bar{U}],v\vDash[p_i]\Kh\alpha_i$, hence $\M[\bar{p}\mapsto\bar{U}],w\vDash \K[p_i]\Kh\alpha_i$. 
    So $\M[\bar{p}\mapsto\bar{U}],w\vDash\K\tensor_n^k(p_1,\cdots,p_n)\land\bigwedge_{i=1}^{n}\K[p_i]\Kh\alpha_i$, which is equivalent to $\M,w\vDash\exists p_1\cdots\exists p_n(\K\tensor_n^k(p_1,\cdots,p_n)\land\bigwedge_{i=1}^{n}\K[p_i]\Kh\alpha_i)$. 
    \item[$\Longleftarrow$] Suppose $\M,w\vDash\exists p_1\cdots\exists p_n(\K\tensor_n^k(p_1,\cdots,p_n)\land\bigwedge_{i=1}^{n}\K[p_i]\Kh\alpha_i)$, then there are $U_i\subseteq W_\M$ such that $\M[\bar{p}\mapsto\bar{U}],w\vDash\K\tensor_n^k(p_1,\cdots,p_n)\land\bigwedge_{i=1}^{n}\K[p_i]\Kh\alpha_i$. 
    
    For the first conjunct: $\M[\bar{p}\mapsto\bar{U}],w\vDash\K\tensor_n^k(p_1,\cdots,p_n)$ means that for any $v\in W_\M$ we have $\M[\bar{p}\mapsto\bar{U}],v\vDash\tensor_n^k(p_1,\cdots,p_n)$. So at least $k$ of $p_i$ is true in $v$, which means that $v$ belongs to at least $k$ of $U_i$. For the second conjunct: $\M[\bar{p}\mapsto\bar{U}],w\vDash\bigwedge_{i=1}^{n}\K[p_i]\Kh\alpha_i$ means that for any $v\in W_\M$ $v\in U_i$ implies that $\R(U_i,\alpha_i)\neq\emptyset$. So, if $U_i\neq\emptyset$, choose an element from $\R(U_i,\alpha_i)$ and denote it as $r_i$. If $U_i=\emptyset$, choose an arbitrary element from $\S(\alpha_i)$ and denote it as $r_i$. 
    
    Combining the meaning of the two conjuncts, we know that for any $v\in W_\M$, $v$ belongs to at least $k$ of $U_i$ and $U_i\neq\emptyset$ implies $r_i\in\R(U_i,\alpha_i)$ for every $i$. Hence, $(r_1,\cdots,r_n)$ is a $n$-tuple such that for any $v\in W_\M$, there are at least $k$ indexes $i\in [1,n]$ such that $r_i\in\R(v,\alpha_i)$, by Proposition \ref{prop.KhTeq}, we have $\M,w\vDash\Kh\tensor^k_n(\alpha_1,\cdots,\alpha_n)$. 
\end{itemize}
\end{proof}
By using the reduction axioms above, all general tensors can be eliminated semantically, and thus $\LKhPT$ and $\LKhP$ are equally expressive. 

Let $\SKhPALPiT$ be $\SKhPALPi$ extended with  $\gt$ and $\Khgt$ for any $n\geq 2$ and $1\leq k\leq n$. Similar to Theorem \ref{thm.comppalkhpi}, it is straightforward to show: 
\begin{theorem}[Soundness and completeness]
Proof system $\SKhPALPiT$ is sound and complete over the class of all models. 
\end{theorem}

\subsection{Support semantics for $\tensor_n^k$}
We can now go back to the support semantics for $\tensor^k_n.$
\begin{definition}[Support for $\tensor_n^k$]
    $\M,s\Vdash \tensor^k_n(\alpha_1,\cdots,\alpha_n)$ iff there exist $n$ subsets $t_1,\cdots,t_n$ of $s$ such that for any $i\in [1,n]$, $\M,t_i\Vdash \alpha_i$ and any $w\in s\subseteq W_\M$ belongs to at least $k$ of $t_i$. 
\end{definition}
The support semantics for other connectives  stays the same as in Definition \ref{support}. 
 Let $\InqL^{\tensor^k_n}$ be the set of valid $\PLT$ formulae by the support semantics. We can show $\KhL^{\tensor^k_n}=\{\alpha\in\PLT | \vDash \Kh\alpha\}$ is exactly $\InqL^{\tensor^k_n}$, based on the following  generalization of Lemma \ref{lem.modelstate}. 
\begin{proposition}\label{prop.modelstateT}
For any $\alpha\in\PLT$, $\M,w\vDash\Kh\alpha\iff \M,W_\M\Vdash\alpha$.

\end{proposition}

\begin{proof}
Based on Lemma \ref{lem.modelstate}, we only consider the case of   $\tensor^k_n(\alpha_1,\cdots,\alpha_n)$ and write $\exists U$ for $\exists U\subseteq W_\M$ for brevity, similarly for $\exists t$. 
$$\begin{aligned}
& \M,w\vDash \Kh(\tensor^k_n(\alpha_1,\cdots,\alpha_n))\\ 
\iff & \M,w\vDash\exists p_1\cdots\exists p_n(\K\tensor_n^k(p_1,\cdots,p_n)\land\bigwedge_{i=1}^{n}\K[p_i]\Kh\alpha_i) \text{ (by Proposition \ref{prop.semantic tensor})}. \\
\iff & \exists U_1,\cdots,U_n, \forall v\in W_\M, v \text{ belongs to at least } k \text{ of } U_i \text{ and } \\
& \qquad \qquad \qquad \qquad \qquad \qquad \qquad  \forall i\in [1, n],   \text{ if } U_i\neq\emptyset \text{ then } \R(U_i,\alpha_i)\not=\emptyset.\\
\iff & \exists t_1,\cdots,t_n, \forall i\in [1,n]\  t_i\Vdash\alpha_i \text{ and }\forall v \in W_\M, v \text{ belongs to at least } k \text{ of } t_i. \\
\iff & \M,W_\M\Vdash\tensor^k_n(\alpha_1,\cdots,\alpha_n). 
\end{aligned}$$
\end{proof}






As shown in \cite{uniform-defi2017}, adding tensor does not increase the expressive power of inquisitive logic. In fact, adding all the general tensors also does not increase the expressive power of inquisitive logic. 

First, we extend the definition of realization in \cite{Ciardelli2011} to our new connectives. 

\begin{definition}[Realizations]\label{def.RR}
\
\begin{itemize}
	\item $\RR(p)=\{p\}$ for $p\in\Prop$
	\item $\RR(\bot)=\{\bot\}$
	\item $\RR(\alpha\vee\beta)=\RR(\alpha)\cup\RR(\beta)$
	\item $\RR(\alpha\wedge\beta)=\{\rho\wedge\sigma \mid \rho\in\RR(\alpha)$ and $\sigma \in\RR(\beta)\}$
	\item $\RR(\alpha\to\beta)=\{\bigwedge_{\rho\in\RR(\alpha)}(\rho\to f(\rho)) \mid f:\RR(\alpha)\to\RR(\beta)\}$
	\item $\RR(\tensor_n^k({\alpha_1,\cdots,\alpha_n}))=\{\neg\bigwedge_{\substack{I\subseteq\{1,2,\cdots,n\}\\ | I |=k}}\neg\bigwedge_{i\in I}\rho_i\mid \text{ for all $i$}, \rho_i\in\RR(\alpha_i)\}$
\end{itemize}
\end{definition}

Then we can generalize the Inquisitive normal form in \cite{Ciardelli2018,Ciardelli2011}. 

\begin{proposition}[Normal form]\label{prop.INF}
For any $\alpha\in\PLT$, $s\Vdash \alpha$ iff $s\Vdash\bigvee_{\rho\in\RR(\alpha)}\rho$. 
\end{proposition}

\begin{theorem}
The languages of $\InqL$ and $\InqL^{\tensor^k_n}$ are equally expressive. 
\end{theorem}
\begin{proof}
By Proposition \ref{prop.INF}, for any $\alpha\in\PLT$, $\alpha$ is equivalent to a disjunction of some $\rho$ without general tensors. 
\end{proof}

In \cite{YangV16}, it is shown that the variants of propositional dependence logics $\PD$, $\PDV$, $\PID$, $\InqL$ are all equally expressive. Similarly, adding general tensors to these logics will also not increase the expressive power. 
\begin{corollary}
Adding general tensors to $\PD$, $\PDV$, $\PID$ or $\InqL$ does not increase their expressive power. 
\end{corollary}

\subsection{Uniform Definability of general tensors}\label{T23}
It is natural to ask whether the generalized tensors are uniformly definable by the standard binary tensor $\tensor$. In \cite{CiardelliB19}, it is proved that $\tensor$ is not uniformly definable in $\InqL$. Inspired by the techniques in \cite{CiardelliB19}, we will show in Theorem \ref{thm.definegt} that
all the $\tensor^k_n$ are \textbf{not} uniformly definable in $\InqL^\tensor$ except $\tensor^1_n$ and $\tensor^n_n$, where $1\leq k \leq n$ and $2\leq n$.

First, we show that $\tensor^n_n$ is a trivial conjunction, $\tensor^1_n$ can be uniformly defined by $\tensor^1_2$, and by using $\top$ or $\bot$, some general tensor can be uniformly defined by others. 

\begin{proposition}\label{prop.TBred}
For any $\alpha_1,\cdots,\alpha_n\in\InqL^{\tensor^k_n}$, there are following properties: 
\begin{itemize}
    \item[(1)] For any $n\geq 2$ and any state $s$, $s\Vdash\tensor^n_n(\alpha_1,\cdots,\alpha_n)\iff s\Vdash\bigwedge_{i=1}^n\alpha_i$. 
    \item[(2)] For any $n\geq 3$ and any state $s$, $s\Vdash\tensor^1_n(\alpha_1,\cdots,\alpha_n)\iff s\Vdash\tensor^1_2(\tensor^1_{n-1}(\alpha_1,\cdots,\alpha_{n-1}),\alpha_n)$. 
    \item[(3)] For any $n\geq 3$, $1\leq k\leq n$ and any state $s$, $s\Vdash\tensor^k_n(\alpha_1,\cdots,\alpha_{n-1},\top)\iff s\Vdash\tensor^{k-1}_{n-1}(\alpha_1,\cdots,\alpha_{n-1})$. 
    \item[(4)] For any $n\geq 3$, $1\leq k\leq n-1$ and any state $s$, $s\Vdash\tensor^k_n(\alpha_1,\cdots,\alpha_{n-1},\bot)\iff s\Vdash\tensor^{k}_{n-1}(\alpha_1,\cdots,\alpha_{n-1})$. 
\end{itemize}
\end{proposition}

\begin{proof}
\begin{itemize}
    \item[(1)] For any $n\geq 2$ and any state $s$, $s\Vdash\tensor^n_n(\alpha_1,\cdots,\alpha_n)$ iff $\exists t_1,\cdots,t_n\subseteq s, \forall i\in[1,n], t_i\Vdash\alpha_i$, and for any $w\in s, w$ belongs to $n$ of $t_i$. So $w$ belongs to all the $t_i$, which means that $t_i=s$ for all $i\in[1,n]$. Hence, for all $i\in[1,n]$ we have $s\Vdash\alpha_i$, which is equivalent to $s\Vdash\bigwedge_{i=1}^n\alpha_i$. 
    \item[(2)] For any $n\geq 3$, $1\leq k\leq n-1$ and any state $s$, $s\Vdash\tensor^1_n(\alpha_1,\cdots,\alpha_n)$ iff $\exists t_1,\cdots,t_n\subseteq s,\forall i\in[1,n], t_i\Vdash\alpha_i$, and $\bigcup_{i=1}^n t_i=s$. Then it is obvious that $\bigcup_{i=1}^{n-1}t_i\Vdash\tensor^1_{n-1}(\alpha_1,\cdots,\alpha_{n-1})$ and $t_n\Vdash\alpha_n$, hence $s\Vdash\tensor^1_2(\tensor^1_{n-1}(\alpha_1,\cdots,\alpha_{n-1}),\alpha_n)$. 
    \item[(3)] For any $n\geq 3$ and any state $s$, $s\Vdash\tensor^k_n(\alpha_1,\cdots,\alpha_{n-1},\top)$ iff $\exists t_1,\cdots,t_{n-1},t_n\subseteq s, \forall i\in[1,n-1], t_i\Vdash \alpha_i$ and $t_n\Vdash\top$, and for any $w\in s$, $w$ belongs to at least $k$ of $t_i$. 
    
    Since $s\Vdash\top$ is trivially true, we can assume $t_n=s$, then the condition is equivalent to $\exists t_1,\cdots,t_{n-1}\subseteq s, \forall i\in[1,n-1], t_i\Vdash \alpha_i$, and for any $w\in s$, $w$ belongs to at least $k-1$ of $t_1,\cdots,t_{n-1}$. Hence, it is equivalent to $s\Vdash\tensor^{k-1}_{n-1}(\alpha_1,\cdots,\alpha_{n-1})$. 
    
    \item[(4)] For any $n\geq 3$, $1\leq k\leq n$ and any state $s$, $s\Vdash\tensor^k_n(\alpha_1,\cdots,\alpha_{n-1},\bot)$ iff $\exists t_1,\cdots,t_{n-1},t_n\subseteq s, \forall i\in[1,n-1], t_i\Vdash \alpha_i$ and $t_n\Vdash\bot$, and for any $w\in s$, $w$ belongs to at least $k$ of $t_i$. 
    
    Since only $\emptyset\Vdash\bot$, so we can assume $t_n=\emptyset$,
    then the condition is equivalent to $\exists t_1,\cdots,t_{n-1}\subseteq s, \forall i\in[1,n-1], t_i\Vdash \alpha_i$, and for any $w\in s$, $w$ belongs to at least $k$ of $t_1,\cdots,t_{n-1}$. Hence, it is equivalent to $s\Vdash\tensor^{k}_{n-1}(\alpha_1,\cdots,\alpha_{n-1})$. 
\end{itemize}
\end{proof}


There are some definitions about uniform definability from \cite{uniform-defi2017} as below. 

\begin{definition}[Context]
A context for a propositional logic $\mathcal{L}$ is an $\mathcal{L}$-formula $\phi(p_1,\cdots,p_n)$ with distinguished atoms $p_1,\cdots,p_n$, and it is also allowed to contain other atoms besides $p_1,\cdots,p_n$. For any $L$-formulae $\psi_1,\cdots,\psi_n$, we write $\phi(\psi_1,\cdots,\psi_n)$ for the formula $\phi(\psi_1/p_1,\cdots,\psi_n/p_n)$. 
\end{definition}

\begin{definition}[Uniform definability]
In a language $\mathcal{L}$, we say that an n-ary connective $\odot$ is uniformly definable if there exists a context $\zeta(p_1,\cdots,p_n)$ such that for all $\chi_1,\cdots,\chi_n\in\mathcal{L}$: $\odot(\chi_1,\cdots,\chi_n)$ is equivalent to $\zeta(\chi_1,\cdots,\chi_n)$. 

\end{definition}

In order to show that $\tensor^2_3$ is not uniformly definable, we consider equivalence relativized to a state $s$. 

\begin{definition}[Relativized equivalence \cite{CiardelliB19}]
Let $s$ be a state in $\M$ and $\phi,\psi\in \PLT$. We say that $\phi$ and $\psi$ are equivalent relativized to $s$, $\phi\equiv_s\psi$ iff for all states $t\subseteq s$, $t\Vdash \phi\iff t\Vdash\psi$.  
\end{definition}
Note that if $\varphi$ and $\psi$ are equivalent then they are equivalent relativized to any state $s$. 

Consider $\psi=p_1\vee p_2\vee p_3\vee p_4$ and $s=\{w_{12},w_{13},w_{14},w_{23},w_{24},w_{34}\}$ where only $p_i,  p_j$ are true in $w_{ij}$ and all of other propositional letters are false. Now, we show that relativized to this state $s$, $\tensor^2_3$ can't be uniformly defined by any context in $\InqL^\tensor$. 

\begin{lemma}\label{lem.T23}
For any context $\phi(p_0)$, with $\phi\in \PL$ not containing $p_1,p_2,p_3,p_4$, $\phi(\psi/p_0)$ would be equivalent to $\bot,\psi,\tensor^1_2(\psi,\psi)$ or $\top$ relativized to $s$. 
\end{lemma}

\begin{proof}


First we notice for any state $t$, $t\Vdash\bot\Rightarrow t\Vdash\psi\Rightarrow t\Vdash\tensor^1_2(\psi,\psi)\Rightarrow t\Vdash\top$ ($\star$).


Then we prove by induction on $\phi$. For short, we write $\phi^*$ for $\phi(\psi/p_0)$: 
\begin{itemize}
    \item For $\phi=\bot$ or $\phi=p$ with $p\neq p_0$: Since we assume that $p_1,p_2,p_3,p_4$ are not in $\varphi$, so $p$ is different from them. Hence, it is obvious that $\phi^*\equiv_s\bot$. 
    \item For $\phi=p_0$: It is obvious that $\phi^*\equiv_s\psi$. 
    \item For $\phi=\phi_1\wedge\phi_2$: so $\phi^*=\phi_1^*\wedge\phi_2^*$ and $t\Vdash\phi_1^*\wedge\phi_2^*$ iff $t\Vdash\phi_1^*$ and $t\Vdash\phi_2^*$. By IH, $\phi_1^*$ and $\phi_2^*$ are both equivalent to one of $\bot,\psi,\tensor^1_2(\psi,\psi)$, $\top$. Since we have ($\star$) and that  $t\Vdash\chi_1\Rightarrow t\Vdash\chi_2$ implies $t\Vdash \chi_1\land\chi_2\iff t\Vdash\chi_1$, it is obvious that $\phi^*$ is also equivalent to one of $\bot,\psi,\tensor^1_2(\psi,\psi)$, $\top$ in $s$. 
    
    \item For $\phi=\phi_1\vee\phi_2$: so $\phi^*=\phi_1^*\vee\phi_2^*$ and $t\Vdash\phi_1^*\vee\phi_2^*$ iff $t\Vdash\phi_1^*$ or $t\Vdash\phi_2^*$. Similarly, we have ($\star$) and  that $t\Vdash\chi_1\Rightarrow t\Vdash\chi_2$ implies $t\Vdash \chi_1\lor\chi_2\iff t\Vdash\chi_2$. Obviously $\phi^*$ is equivalent to one of $\bot,\psi,\tensor^1_2(\psi,\psi)$, $\top$ in $s$. 
    \item For $\phi=\phi_1\to\phi_2$: so $\phi^*=\phi_1^*\to\phi_2^*$, and $t\Vdash\phi_1^*\to\phi_2^*$ iff for any $t'\subseteq t$, $t'\Vdash\phi_1^*$ implies $t'\Vdash\phi_2^*$. Since we have ($\star$), we could know that: 
    \begin{itemize}
        \item $\bot\to\bot$, $\bot\to\psi$, $\bot\to\tensor^1_2(\psi,\psi)$, $\bot\to\top$, $\psi\to\psi$, $\psi\to\tensor^1_2(\psi,\psi)$, $\psi\to\top$, $\tensor^1_2(\psi,\psi)\to\tensor^1_2(\psi,\psi)$, $\tensor^1_2(\psi,\psi)\to\top$ and $\top\to\top$ are all equivalent to $\top$ in $s$. Also, if $\psi\lra\bot$, then $\psi\to\bot$ and  $\tensor^1_2(\psi,\psi)\to\bot$ are equivalent to $\top$ in $s$.
        \item If $\psi\not\lra\bot$, $\psi\to\bot$, $\tensor^1_2(\psi,\psi)\to\bot$ and $\top\to\bot$ are all equivalent to $\bot$ in $s$. 
        \item $\tensor^1_2(\psi,\psi)\to\psi$ and $\top\to\psi$ are equivalent to $\psi$ in $s$.
        \item $\top\to\tensor^1_2(\psi,\psi)$ is equivalent to $\tensor^1_2(\psi,\psi)$ in $s$. 
    \end{itemize}
    Hence, $\varphi^*$ is equivalent to one of $\bot,\psi,\tensor^1_2(\psi,\psi)$, $\top$ in $s$. 
    \item For $\phi=\tensor^1_2(\phi_1,\phi_2)$: so $\phi^*=\tensor^1_2(\phi_1^*,\phi_2^*)$. We consider the following cases: 
    \begin{itemize}
        \item $\phi_1^*\equiv_s\top$. Then $\tensor^1_2(\phi_1^*,\phi_2^*)\equiv_s \top$. 
        \item $\phi_1^*\equiv_s\bot$. Then $\tensor^1_2(\phi_1^*,\phi_2^*)\equiv_s \phi_2^*$. 
        \item $\phi_1^*\equiv_s\psi$. If $\phi_2^*\equiv_s\top$ or $\phi_2^*\equiv_s\bot$, it would be the same as former cases. Then we need to distinguish two sub-cases:
        \begin{itemize}
            \item $\phi_2^*\equiv_s\psi$. Then $\tensor^1_2(\phi_1^*,\phi_2^*)\equiv_s\tensor^1_2(\psi,\psi)$. 
            \item $\phi_2^*\equiv_s\tensor^1_2(\psi,\psi)$. Then $t\Vdash\phi^*$ $\iff$ there are $t_1,t_2\subseteq t$ and $t_1\cup t_2=t$ such that $t_1\Vdash\psi$ and $t_2\Vdash\tensor^1_2(\psi,\psi)$ $\iff$ there are $t_1,t_2\subseteq t$, $t_1\cup t_2=t$ and $p_{i_1},p_{i_2},p_{i_3}$ such that $p_{i_1}$ is true in any $w\in t_1$ and for any $w\in t_2$, $p_{i_2}$ or $p_{i_3}$ is true in $w$ $\iff$ there are $p_{i_1},p_{i_2},p_{i_3}$ such that for any $w\in t$, $p_{i_1}$, $p_{i_2}$ or $p_{i_3}$ is true in $w$. However, there are only four propositional letters $p_1,p_2,p_3,p_4$ and in each $w\in s$, two of these propositional letters are true. So consider $p_1,p_2$ and $p_3$, we will know that for any $w\in t\subseteq s$, at least one of $p_1,p_2$ and $p_3$ is true in $w$. Hence, $\tensor^1_2(\psi,\tensor^1_2(\psi,\psi))\equiv_s\top$. 
        \end{itemize}
        \item $\phi_1^*\equiv_s\tensor^1_2(\psi,\psi)$. Then if $\phi_2^*\equiv_s\top$, $\phi_2^*\equiv_s\bot$ or $\phi_2^*\equiv_s\psi$, it would be the same as former cases. And if $\phi_2^*\equiv_s\tensor^1_2(\psi,\psi)$, the proof is similar to the previous case and the result is that $\tensor^1_2(\tensor^1_2(\psi,\psi),\tensor^1_2(\psi,\psi))\equiv_s\top$. 
    \end{itemize}
\end{itemize}
\vspace{-10pt}
\end{proof}
\begin{lemma}
$\tensor^2_3$ is not uniformly definable in $\InqL^\tensor$. 
\end{lemma}

\begin{proof}
If $\tensor^2_3$ is uniformly definable in $\InqL^\tensor$, there will be a context $\phi(p)$ such that for any $\chi\in\InqL^\tensor$: $\phi(\chi)$ is equivalent to $\tensor^2_3(\chi,\chi,\chi)$. 

However, as we proved in Lemma \ref{lem.T23}, 
for any context $\phi(p_0)\in\InqL^\tensor$, $\phi(\psi/p_0)$ would be equivalent to $\bot$, $\psi$, $\tensor^1_2(\psi,\psi)$ or $\top$ relativized to s. But it is obvious that $\tensor^2_3(\psi,\psi,\psi)$ is not equivalent to $\bot$, $\psi$, $\tensor^1_2(\psi,\psi)$ or $\top$ relativized to $s$. Hence, $\tensor^2_3(\psi,\psi,\psi)$ and $\phi(\psi/p_0)$ are not equivalent relativized to $s$, and hence not equivalent in general, which gives rise to a contradiction! 
\end{proof}


\begin{theorem}\label{thm.definegt}
All the $\tensor^k_n$ are \textbf{not} uniformly definable in $\InqL^\tensor$ except $\tensor^1_n$ and $\tensor^n_n$,  i.e., for any $2\leq k\leq n-1$, $\tensor^k_n$ is not uniformly definable.
\end{theorem}

\begin{proof}
When $2\leq k\leq n-1$ (thus $n\geq 3$), by Proposition \ref{prop.TBred}, $\tensor_3^2$ can be uniformly defined by $\tensor^k_n$ in the way of fixing some components as $\top$ or $\bot$, so $\tensor^2_3$ is not uniformly definable in $\InqL^\tensor$ implies that $\tensor^k_n$ is not uniformly definable in $\InqL^\tensor$. 
\end{proof}

\section{Conclusions and future work}

In this paper, we proposed an epistemic interpretation of the tensor disjunction in dependence logic. The interpretation is inspired by the notion of weak disjunction in Medvedev's early work in terms of the BHK-like semantics. The connection between the two disjunctions is exposed in inquisitive logic with tensor disjunction, studied in the literature. We introduce a powerful dynamic epistemic language in which the corresponding know-how formulae of each $\InqL^\tensor$ formula can be formulated and reduced to a know-how free formula. In particular, the tensor disjunction can be defined by an epistemic formula using propositional quantifiers. We give the axiomatization of our full logic, and generalize the tensor disjunction to a family of $n$-ary operators parametered by a $k\leq n$, which capture the intuitive epistemic situations that one knows a list of $n$ possible answers to $n$ questions such that $k$ of the $n$ answers are correct.  

Besides further technical questions regarding our logic, the generalized tensors particularly invite further investigations. Its obvious combinatorial features may find applications in cryptographic protocols and game theory. To see the connection with the latter, we end the paper with the following interesting scenario where $\tensor^2_3$ makes perfect sense. Consider a badminton match between two teams. Each team has one good player and two other less capable ones. We can measure the abilities of the players by numbers, which will determine the result of the matches in the most obvious way. For team $A$, it is $6,2,2$ for the three players, and for team $B$ it is $5,3,3$. The battle between the two teams consists of three single matches, and the rule of game does not prevent one player from playing two matches if not in a row, although the second time the player will lose $1/3$ of his or her ability due to tiredness. Now, with some reflection, we can see team $B$ has a unique arrangement of the playing players to make sure they can win at least two out of the three matches no matter how team $A$ orders their playing  players. Do you know which one?


\bibliographystyle{aiml22}
\bibliography{name}

\end{document}